\documentclass[letterpaper,final,twoside,leqno]{amsart}

\usepackage[normalem]{ulem}
\usepackage{mathtools}
\usepackage{mathrsfs}  
\usepackage[arrow,curve,matrix]{xy}

\usepackage{scalerel}
\usepackage{comment}
\usepackage{xspace}
\usepackage{multicol}

\usepackage{array}
\newcolumntype{P}[1]{>{\centering\arraybackslash}p{#1}}

\newif\ifdraft 
\draftfalse

\numberwithin{figure}{section}

\usepackage{amsmath,amsfonts,amsthm,mathrsfs}
\usepackage{amssymb}
\usepackage{mathrsfs}

\DeclareFontFamily{OMS}{rsfs}{\skewchar\font'60}
\DeclareFontShape{OMS}{rsfs}{m}{n}{<-5>rsfs5 <5-7>rsfs7 <7->rsfs10 }{}
\DeclareSymbolFont{rsfs}{OMS}{rsfs}{m}{n}
\DeclareSymbolFontAlphabet{\scr}{rsfs}

\usepackage[bookmarks]{hyperref}
\usepackage[usenames,dvipsnames]{xcolor}
\hypersetup{
    colorlinks=true,
    linkcolor={blue!50!black},
    citecolor={green!50!black},
    urlcolor={red!80!black}
}

\usepackage[alphabetic,initials]{amsrefs}

\usepackage{enumitem}

\usepackage{chngcntr}

\ifdraft
\usepackage[notcite,notref,color]{showkeys}

\definecolor{labelkey}{gray}{0.5}
\fi

\numberwithin{figure}{section}

\DeclareMathOperator{\Gal}{Gal}
\DeclareMathOperator{\gmin}{gmin}
\DeclareMathOperator{\X}{\mathfrak{X}}
\DeclareMathOperator{\Ob}{Ob}
\DeclareMathOperator{\pr}{pr}
\DeclareMathOperator{\rank}{rank}

\DeclareMathOperator{\Pic}{Pic}

\DeclareMathOperator{\Spec}{Spec}

\DeclareMathOperator{\sa}{sa}
\DeclareMathOperator{\Var}{Var}

\newcommand{\h}{\overline{h}}

\newcommand{\wtilde}{\widetilde}

\newcommand{\hooklongrightarrow}{\lhook\joinrel\longrightarrow}

\newcommand{\theoremref}[1]{\hyperref[#1]{Theorem~\ref*{#1}}}
\newcommand{\lemmaref}[1]{\hyperref[#1]{Lemma~\ref*{#1}}}
\newcommand{\definitionref}[1]{\hyperref[#1]{Definition~\ref*{#1}}}
\newcommand{\propositionref}[1]{\hyperref[#1]{Proposition~\ref*{#1}}}
\newcommand{\conjectureref}[1]{\hyperref[#1]{Conjecture~\ref*{#1}}}
\newcommand{\corollaryref}[1]{\hyperref[#1]{Corollary~\ref*{#1}}}
\newcommand{\exampleref}[1]{\hyperref[#1]{Example~\ref*{#1}}}

\makeatletter
\let\old@caption\caption
\renewcommand*{\caption}[1]{%
  \setcounter{figure}{\value{equation}}%
  \stepcounter{equation}%
  \old@caption{#1}\relax%
}
\makeatother

\newcommand{\parref}[1]{\hyperref[#1]{\S\ref*{#1}}}

\makeatletter
\newcommand*\if@single[3]{%
  \setbox0\hbox{${\mathaccent"0362{#1}}^H$}%
  \setbox2\hbox{${\mathaccent"0362{\kern0pt#1}}^H$}%
  \ifdim\ht0=\ht2 #3\else #2\fi
  }
\newcommand*\rel@kern[1]{\kern#1\dimexpr\macc@kerna}
\newcommand*\widebar[1]{\@ifnextchar^{{\wide@bar{#1}{0}}}{\wide@bar{#1}{1}}}
\newcommand*\wide@bar[2]{\if@single{#1}{\wide@bar@{#1}{#2}{1}}{\wide@bar@{#1}{#2}{2}}}
\newcommand*\wide@bar@[3]{%
  \begingroup
  \def\mathaccent##1##2{%
    \if#32 \let\macc@nucleus\first@char \fi
    \setbox\z@\hbox{$\macc@style{\macc@nucleus}_{}$}%
    \setbox\tw@\hbox{$\macc@style{\macc@nucleus}{}_{}$}%
    \dimen@\wd\tw@
    \advance\dimen@-\wd\z@
    \divide\dimen@ 3
    \@tempdima\wd\tw@
    \advance\@tempdima-\scriptspace
    \divide\@tempdima 10
    \advance\dimen@-\@tempdima
    \ifdim\dimen@>\z@ \dimen@0pt\fi
    \rel@kern{0.6}\kern-\dimen@
    \if#31
      \overline{\rel@kern{-0.6}\kern\dimen@\macc@nucleus\rel@kern{0.4}\kern\dimen@}%
      \advance\dimen@0.4\dimexpr\macc@kerna
      \let\final@kern#2%
      \ifdim\dimen@<\z@ \let\final@kern1\fi
      \if\final@kern1 \kern-\dimen@\fi
    \else
      \overline{\rel@kern{-0.6}\kern\dimen@#1}%
    \fi
  }%
  \macc@depth\@ne
  \let\math@bgroup\@empty \let\math@egroup\macc@set@skewchar
  \mathsurround\z@ \frozen@everymath{\mathgroup\macc@group\relax}%
  \macc@set@skewchar\relax
  \let\mathaccentV\macc@nested@a
  \if#31
    \macc@nested@a\relax111{#1}%
  \else
    \def\gobble@till@marker##1\endmarker{}%
    \futurelet\first@char\gobble@till@marker#1\endmarker
    \ifcat\noexpand\first@char A\else
      \def\first@char{}%
    \fi
    \macc@nested@a\relax111{\first@char}%
  \fi
  \endgroup
}
\makeatother

\usepackage{amsbsy}

\usepackage{marvosym}

\DeclareMathAlphabet{\smallchanc}{OT1}{pzc}%
                                 {m}{it}
\DeclareFontFamily{OT1}{pzc}{}
\DeclareFontShape{OT1}{pzc}{m}{it}%
             {<-> s * [1.100] pzcmi7t}{}
\DeclareMathAlphabet{\mathchanc}{OT1}{pzc}%
                                 {m}{it}

\newcommand{\sB}{\mathscr{B}}

\newcommand{\sF}{\mathscr{F}}
\newcommand{\sG}{\mathscr{G}}

\newcommand{\sL}{\mathscr{L}}
\newcommand{\sM}{\mathscr{M}}

\newcommand{\sO}{\mathscr{O}}

\newcommand{\sV}{\mathscr{V}}

\newcommand{\bC}{\mathbb{C}}

\newcommand{\bN}{\mathbb{N}}

\newcommand{\bP}{\mathbb{P}}
\newcommand{\bQ}{\mathbb{Q}}
\newcommand{\bR}{\mathbb{R}}

\newcommand{\bZ}{\mathbb{Z}}

\newcommand{\frM}{\mathfrak{M}}

\DeclareSymbolFont{largesymbolsA}{U}{jkpexa}{m}{n}
\SetSymbolFont{largesymbolsA}{bold}{U}{jkpexa}{bx}{n}
\DeclareMathSymbol{\varprod}{\mathop}{largesymbolsA}{16}
\makeatletter
\newcommand{\LeftEqNo}{\let\veqno\@@leqno}
\makeatother
\newcommand{\properideal}%
        {\subsetneq}

\newcommand\dash[1]{\rule[-.2ex]{#1}{.4pt}}

\DeclareMathOperator{\Aut}{Aut}

\DeclareMathOperator{\codim}{codim}

\DeclareMathOperator{\Hilb}{Hilb}

\DeclareMathOperator{\id}{{id}}

\DeclareMathOperator{\reg}{reg}

\DeclareMathOperator{\supp}{{supp}}

\newcommand{\factor}[2]{\left. \raise .2em\hbox{\ensuremath{#1}\vphantom{$I^d$}}
\hskip -.1em \right/ \hskip -.4em \raise -.3em\hbox{\ensuremath{#2}}}%
\newcommand\mtimes[3]{{\varprod_{#1}^{#2}}_{\raise 1ex \hbox{\scriptsize #3}}}%

\newcommand{\sblank}{\dash{.6em}}

\def\dimcoh#1.#2.#3.{h^{#1}(#2,#3)}
\def\hypcoh#1.#2.#3.{\mathbb H_{\vphantom{l}}^{#1}(#2,#3)}
\def\loccoh#1.#2.#3.#4.{H^{#1}_{#2}(#3,#4)}
\def\dimloccoh#1.#2.#3.#4.{h^{#1}_{#2}(#3,#4)}
\def\lochypcoh#1.#2.#3.#4.{\mathbb H^{#1}_{#2}(#3,#4)}
\def\seslong#1.#2.#3.{0  \longrightarrow  #1   \longrightarrow 
 #2 \longrightarrow #3 \longrightarrow 0} 
\def\sesshort#1.#2.#3.{0
 \rightarrow #1 \rightarrow #2 \rightarrow #3 \rightarrow 0}
\def\Iff#1#2#3{
\hfil\hbox{\hsize =#1
\vtop{\noin #2}
\hskip.5cm 
\lower.5\baselineskip\hbox{$\Leftrightarrow$}\hskip.5cm
\vtop{\noin #3}}\hfil\medskip}
\newcommand{\union}\cup
\newcommand{\intersect}\cap
\newcommand{\Union}\bigcup
\newcommand{\Intersect}\bigcap
\def\myoplus#1.#2.{\underset #1 \to {\overset #2 \to \oplus}}

\newcommand{\resto}[1]{\raise -.5ex\hbox{$\vert$}_{#1}}

\newcommand\noin{\noindent}

   [2017/06/13 v0.1 finetuning mabliautoref and theorem setup]

\RequirePackage{hyperref}

\newcommand{\sectionsize}{} 
\newcommand{\theoremsize}{} 

\renewcommand{\subsectionautorefname}{\sectionsize\sf \subsectionautorefname}

\makeatletter
\@ifdefinable\equationname{\let\equationname\equationautorefname}
\def\equationautorefname~#1\@empty\@empty\null{\protect{\theoremsize
    (#1\@empty\@empty\null)}}%
\@ifdefinable\AMSname{\let\AMSname\AMSautorefname}
\def\AMSautorefname~#1\@empty\@empty\null{
  ( #1\@empty\@empty\null)}%
\@ifdefinable\itemname{\let\itemname\itemautorefname}
\def\itemautorefname~#1\@empty\@empty\null{\theoremsize{%
    {#1}}\@empty\@empty\null%
}%
\makeatother

\RequirePackage{amsthm}
\RequirePackage{aliascnt}
\newcommand{\basetheorem}[3]{%
    \newtheorem{#1}{#2}[#3]
    \newtheorem*{#1*}{#2}
    \expandafter\def\csname #1autorefname\endcsname{#2}
}%
\newcommand{\maketheorem}[3]{%
    \newaliascnt{#1}{#2}
    \newtheorem{#1}[#1]{\theoremsize #3}
    \aliascntresetthe{#1}
    \expandafter\def\csname #1autorefname\endcsname{#3}
    \newtheorem{#1*}{#3}
}%
\newcommand{\baseremark}[3]{%
    \newtheorem{#1}{#2}{#3}
    \newtheorem*{#1*}{#2}
    \expandafter\def\csname #1autorefname\endcsname{#2}
}%
\newcommand{\makeremark}[3]{%
    \newaliascnt{#1}{#2}
    \newtheorem{#1}[equation]{#3}
    \aliascntresetthe{#1}
    \expandafter\def\csname #1autorefname\endcsname{\theoremsize\sf #3}
    \newtheorem{#1*}{#3}
}%
\theoremstyle{plain}   

\basetheorem{theorem}{Theorem}{section}

\theoremstyle{definition}    

\newcommand{\hide}[1]{}

\setlist[enumerate, 1]{font=\upshape}

\numberwithin{figure}{section}

\theoremstyle{plain}
\maketheorem{prop}{theorem}{Proposition}
\maketheorem{reform}{theorem}{Reformulation}
\maketheorem{cor}{theorem}{Corollary}
\maketheorem{lem}{theorem}{Lemma}
\maketheorem{conj}{theorem}{Conjecture}

\newtheorem{proposition}[prop]{Proposition}
\newtheorem{reformulation}[reform]{Reformulation}
\newtheorem{corollary}[cor]{Corollary}
\newtheorem{lemma}[lem]{Lemma}

\theoremstyle{definition}
\maketheorem{defini}{theorem}{Definition}
\newtheorem{definition}[defini]{Definition}

\maketheorem{defff}{theorem}{Definition-Theorem}
\newtheorem{def-thm}[defff]{Definition-Theorem}
\theoremstyle{remark}
\maketheorem{rem}{theorem}{Remark}
\newtheorem{remark}[rem]{Remark}
\theoremstyle{example}
\maketheorem{exam}{theorem}{Example}
\newtheorem{example}[exam]{Example}
\maketheorem{notat}{theorem}{Notation}
\maketheorem{notr}{theorem}{Notation-Remark}
\maketheorem{defnot}{theorem}{Definition-Notation}
\newtheorem{notation}[notat]{Notation}
\newtheorem{not-rem}[notr]{Notation-Remark}
\newtheorem{def-not}[defnot]{Definition-Notation}

\maketheorem{set}{theorem}{Set-up}
\newtheorem{set-up}[set]{Set-up}

\maketheorem{term}{theorem}{Terminology}
\newtheorem{terminology}[term]{Terminology}

\maketheorem{obs}{theorem}{Observation}
\newtheorem{observation}[obs]{Observation}

\maketheorem{fac}{theorem}{Fact}
\newtheorem{fact}[fac]{Fact}

\maketheorem{ass}{theorem}{Assumption}
\newtheorem{assumption}[ass]{Assumption}

\maketheorem{conn}{theorem}{Convention}

\maketheorem{clai}{theorem}{Claim}
\newtheorem{claim}[clai]{Claim}

\maketheorem{sclai}{theorem}{Subclaim}
\newtheorem{subclaim}[sclai]{Subclaim}

\setlist[enumerate]{label=(\thetheorem.\arabic*), before={\setcounter{enumi}{\value{equation}}}, after={\setcounter{equation}{\value{enumi}}}}

\numberwithin{equation}{theorem}

\makeatletter
\let\amsmath@bigm\bigm

\renewcommand{\bigm}[1]{%
  \ifcsname fenced@\string#1\endcsname
    \expandafter\@firstoftwo
  \else
    \expandafter\@secondoftwo
  \fi
  {\expandafter\amsmath@bigm\csname fenced@\string#1\endcsname}%
  {\amsmath@bigm#1}%
}
\newcommand{\DeclareFence}[2]{\@namedef{fenced@\string#1}{#2}}
\makeatother

\DeclareFence{\mid}{|}

\begin{document}

\title[Boundedness results for smooth families of varieties]
{Boundedness results for families of non-canonically polarized projective varieties}

\author{Kenneth Ascher} 
\address{Kenneth Ascher, Department of Mathematics, University of California Irvine,
Irvine, CA 92697,  USA}
\email{\href{mailto:kascher@uci.edu }{kascher@uci.edu}}
\urladdr{\href{https://www.math.uci.edu/~kascher/}
  {https://www.math.uci.edu/~kascher/}}

\author{Behrouz Taji} \address{Behrouz Taji, School of Mathematics and Statistics,
The University of New South Wales Sydney, NSW 2052 Australia}
\email{\href{mailto:b.taji@unsw.edu.au}{b.taji@unsw.edu.au}}
\urladdr{\href{https://web.maths.unsw.edu.au/~btaji//}
  {https://web.maths.unsw.edu.au/~btaji/}}


\keywords{Families of manifolds, flat families, boundedness of morphisms, stacks and moduli problems}

\subjclass[2020]{14D06, 14D22, 14D23, 14E05, 14D07, 14E30, 14J32}


\setlength{\parskip}{0.19\baselineskip}


\begin{abstract}
We prove that, over a smooth quasi-projective curve, 
the set of non-isotrivial, smooth and projective families of 
polarized varieties with a fixed Hilbert polynomial
and semi-ample canonical bundle is bounded.  
This extends the boundedness results of Arakelov, Parshin, and Kov\'acs--Lieblich beyond the canonically polarized case.
%
\end{abstract}

\maketitle


\newcommand\hmarginpar[1]{}

\section{Introduction and main results}
\label{Section1-Introduction}

Bounding classes of varieties of a fixed type is a recurring theme in many problems in 
algebraic geometry. In particular, bounding families of varieties of fixed type
over a fixed base can be traced back 
to the celebrated works of Parshin \cite{Parshin68} and Arakelov \cite{Arakelov71} in the setting of 
isomorphism classes of smooth families of projective 
curves of genus $g \geq 2$ over a curve.  
This boundedness result was successfully generalized to arbitrary dimension for smooth projective families 
of canonically polarized varieties by the combined works of 
Bedulev-Viehweg~\cite{Bedulev-Viehweg00} 
and Kov\'acs--Lieblich \cite{KoL10}. As pointed out in Viehweg's survey on Arakelov inequalities \cite[pp. 247--248]{Vie08}, there are much fewer results known if the canonically polarized assumption is dropped, in particular, the analogous boundedness statement is unknown. As such,  we take this as our aim in this article: we analyze boundedness problems for smooth projective families of 
complex varieties that are not canonically polarized, for example (polarized) families of 
varieties with only semi-ample canonical bundle.

In this context, over curves, our main 
result is the following theorem. 

\begin{theorem}\label{thm:curve}
Let $C^0$ be any smooth quasi-projective curve. 
The set of non-isotrivial, smooth, polarized and projective families of varieties with semi-ample canonical bundle 
over $C^0$, with Hilbert polynomial $\h$, 
is bounded.
\end{theorem}

See  \autoref{sec:modulifunctors} for the definition of $\h$ and the relevant moduli functor. The families in \autoref{thm:curve} are associated to Viehweg's moduli functor for isomorphism classes of 
smooth polarized varieties with semi-ample canonical line bundle see \autoref{def:sa} (c.f. \cite{Viehweg95}).
As mentioned above, if we assume that the geometric fibers are canonically polarized, \autoref{thm:curve} is the celebrated 
result of Parshin \cite[Thm.~1]{Parshin68} and Arakelov \cite{Arakelov71}, when fibers are of dimension one, 
and Kov\'acs--Lieblich~\cite[Thm.~1.6]{KoL10}, 
for fibers of arbitrary dimension. 

The notion of boundedness is defined as follows (see \autoref{PolB} for a more precise definition in this context).
\begin{definition}\label{def:BB}
Let $V$ be a smooth quasi-projective variety. We call a set $\mathfrak C_V$ of projective morphisms over $V$ bounded 
(resp. birationally bounded), if there are  schemes $W$ and $\mathcal Y$ of finite type together with a morphism  
 $\mathfrak f:\mathcal Y\to W\times V$ of finite type such that for every $f^0: U\to V\in\mathfrak C_V$
there is a closed point $w\in W$ for which we have $\mathcal Y_{ \{ w \} \times V } \simeq_V U$ (resp. 
$\mathcal Y_{ \{ w \} \times V} \overset{\rm{bir}}{\sim}_V U$). 
\end{definition}
%
%


In the canonically polarized case, it was shown by \cite{KoL10} that \autoref{thm:curve} in fact 
holds over base schemes of arbitrary dimension. 
However we can only partially extend \autoref{thm:curve} to higher dimensional base spaces. 
More precisely, we prove \autoref{thm:curve} holds over an arbitrary base space as long as 
the variation in the family is maximal
(see subsection \ref{def:var}). 
%
%
\begin{theorem}\label{main:maxvar}
Let $V$ be any smooth quasi-projective variety. 
The set of smooth, polarized and projective families of varieties with semi-ample canonical bundle 
over $V$, with Hilbert polynomial $\h$ and maximal variation, 
is bounded.
\end{theorem}


\subsection{Background and an overview of the strategy of the proof.}

In the canonically polarized case, the proof of \autoref{thm:curve}
consisted of two main steps:

\begin{enumerate}
\item\label{CB} \emph{(Bounding coarse moduli maps.)} The set of morphisms 
to the associated coarse moduli space 
that arise from smooth canonically polarized families with a fixed Hilbert polynomial
are bounded (in the sense of  \autoref{def:CB}). 
This is a consequence of Bedulev--Viehweg~\cite{Bedulev-Viehweg00}, 
Koll\'ar \cite{Kollar90}, Viehweg \cite{Vie10}  
and Kov\'acs--Lieblich \cite[Sect.~2]{KoL10} (see also Viehweg--Zuo~\cite[\S6]{VZ02}). 
\smallskip
\item\label{B}  \emph{(Bounding families.)} The set of isomorphism classes of the 
families is bounded  \cite[Thm.~1.6]{KoL10} 
as a consequence of a general boundedness result \cite[Thm.~1.7]{KoL10} for any 
\emph{compactifiable} Deligne--Mumford stack
whose coarse moduli maps can be bounded in the sense of \autoref{CB}. 
\end{enumerate}
%
%
%
%
\begin{definition}[Compactifable stacks]\label{def:compact}
A separated Deligne--Mumford stack of finite type is said to be compactifiable 
if it has a quasi-projective coarse moduli space and an 
open immersion into a Deligne--Mumford stack that is proper over $\Spec \mathbb C$. 
\end{definition}

Kresch proved in \cite[Thm. 5.3]{Kresch} that a separated Deligne--Mumford stack of finite type is compactifiable if it is isomorphic to 
a quotient stack by a linear group
and with a quasi-projective coarse 
space. We note that Rydh has more general results on compactifiability of Deligne-Mumford stacks \cite{Rydh}. Classical examples are the stack of curves $\mathcal M_g$ of genus $g$ \cite{DM69} and that of canonically 
polarized projective manifolds $\mathcal M_h$ 
with fixed Hilbert polynomial $h$ (\cite[Lem.~6.1]{KoL10}).

Equally important in the context of boundedness problems, 
in addition to compactifiabilty, is the fact that 
the formation of the coarse space of the latter two moduli stacks 
is naturally \emph{compatible} with the 
birational variation of the parameterized objects (\ref{def:var}), i.e.,
the latter is  equal to the dimension of the image of the 
associated coarse moduli map (see \autoref{def:comp} for the precise 
definition). This turns out to be one of the main obstacles in generalizing \cite{KoL10} to the non-canonically polarized case, and we discuss this now.


\subsubsection{Background in the semi-ample case.}
\label{background-SemiAmple}
For families with semi-ample canonical it is 
conceivable at first that one can generalize \autoref{CB} and \autoref{B} in the same manner as in 
the canonically polarized case, 
using Viehweg's moduli functor $\mathfrak M_{\h}^{\sa}$ (\autoref{def:sa}). 
After all, the coarse space $M_{\h}^{\sa}$ 
of the latter is also constructed as a quotient  
by a linear group and thus one can readily note, 
using Kresch's criterion and Viehweg's original constructions, 
that the associated stack is compactifiable (see \autoref{prop:compact}). 
Proving \autoref{CB} for $M_{\h}^{\sa}$ however turns out to be considerably more difficult. 
The problem largely lies in the incompatibility of $M_{\h}^{\sa}$ with the 
(birational) notion of variation in the sense of \autoref{def:comp}.

Roughly speaking, as Viehweg also explains in \cite[p. 247]{Vie08},  
\autoref{CB} requires strict positivity of $\det (f_*\omega^m_{X/B})$ for a smooth 
compactification $f: X\to B$ of relevant families, and for large enough $m$. 
This in turn relies on the notion of variation (\ref{def:var}). 
As the latter is birational in nature, it is not, in general, consistent 
with the construction of moduli spaces that parametrize 
isomorphism classes of polarized varieties.
For example, a birationally isotrivial family with trivial $\det(f_*\omega^m_{X/B})$ 
may not get contracted by the associated moduli map,
and may even be generically finite over the coarse space. 
In this situation it becomes impossible to trace any connection between 
$\det (f_*\omega^m_{X/B})$ and an ample line bundle in the coarse space, 
which is classically used to establish \autoref{CB} using \emph{Arakelov inequalities}.
In \cite{Bedulev-Viehweg00} and \cite{KoL10} this is reflected in the use of 
 Koll\'ar--Shepherd-Barron and Alexeev's compactification of the  
 coarse moduli as a projective scheme and descent of $\det (f_*\omega^m_{X/B})$
 to a universal ample line bundle in this space.

There is however another moduli functor $\mathfrak P_{\h}^{\sa}$ 
that Viehweg constructs in \cite[\S7.6]{Viehweg95}, which thanks to \cite{Kawamata85}, 
\emph{is} compatible with variation. 
Unfortunately, the construction of 
its coarse moduli space $P_{\h}$, which is proved to be 
quasi-projective \cite[Thm.~1.14]{Viehweg95}, 
requires the identification of orbits of a certain non-linear
group action on $M_{\h}^{\sa}$, yielding a moduli of numerical 
equivalence classes \cite[p.~16]{Viehweg95}, 
as opposed to a moduli of isomorphism classes 
of polarized schemes (\ref{SS:functors}). 
Multiple problems arise from this construction vis-a-vis application to boundednes problems, e.g.
isomorphisms between the parametrized objects of $\mathfrak P_{\h}^{\sa}$ are not in the sense of polarized schemes, 
a property that plays a prominent role in proving \autoref{CB} in the canonically polarized case
(see for example \autoref{fact:LiftIsom}).  

\smallskip 
 
     \begin{tabular}{ |P{1cm}|P{9cm}|P{1.7cm}|}
 \hline
 \multicolumn{3}{|c|}{\textbf{Table 1}: \;\;\;\;\ \;\;\;\;\; \;\;\;\;\; Viehweg's moduli functors \;\;\;\;\;\;   \;\;\;\;\; \;\;\;  \;\;\;\;\; \;\;\;\;\;\;\;  coarse moduli} \\
 \hline
 \textbf{$\mathfrak M_{\h}^{\sa}$}     & moduli functor of polarized manifolds with semi-ample canonical bundle, \autoref{def:sa}  & $M_{\overline h}^{\sa}$ \\
  \hline
 \textbf{$\mathfrak P_{\h}^{\sa}$}     & moduli functor of polarized manifolds with semi-ample canonical bundle up to numerical equivalence, \autoref{2}, \autoref{rem:P} &  $P^{\sa}_{\h}$  \\
 \hline
 \end{tabular}

\smallskip 


This forms the backdrop for \autoref{thm:curve}; 
that for smooth projective families over curves with fibers having semi-ample canonical bundle, 
one can 
skirt the problem of the absence of 
a compactifiable associated stack that is compatible with the notion 
of variation, and still establish boundedness. To do so we use a different parametrizing strategy. 
\subsubsection{An alternative parameterization strategy.}
The first key observation is that
for any moduli functor of varieties with good minimal models
that is \emph{representable} by a quasi-projective variety (i.e., is equipped 
with a universal family) one can resolve the above 
compatibility issue by using a suitable choice of a polarization 
defined for the universal family
(see \autoref{obs:change}). 

Guided by this observation, instead of working directly with $\frM_{\h}^{\sa}$, we consider 
a different functor (\autoref{ss:Hilb2}) that in some sense \emph{approximates}
$\frM^{\sa}_{\h}$ but has the benefit of being 
representable by a universal family over a quasi-projective scheme
(\autoref{claim:good}). We refer to this new functor as 
 \emph{Hilbert functor of doubly embedded schemes}. 
We then prove \autoref{CB} for the coarse moduli space of this latter functor
(\autoref{Section3-WB}). 
Through a process of pullback (\autoref{lem:Factors}) and descent (\autoref{prop:CB} and \autoref{eq:Diag}), we use this 
to prove \autoref{CB} for $M^{\sa}_{\h}$ itself -- this is 
where the additional assumption of maximal variation, or non-isotriviality in the curve case, is needed. 

We point out that in the canonically polarized case, \autoref{CB} is a consequence of 
the notion of \emph{weak-boundedness} for the associated stack 
(see \autoref{WBGeometry}). Although we do not prove that the stack of $\mathfrak M^{\sa}_{\h}$ 
is weakly-bounded, we show that, at least for maximally-varying families, 
it satisfies boundedness properties in the sense of \autoref{CB}.
With the associated stack of $M_{\h}^{\sa}$ now being compactifiable (\autoref{prop:compact}), we then 
follow the strategy of \cite{KoL10} (see \autoref{thm:KL}).

We note that weak-boundedness for the moduli stack of smooth projective 
varieties $Y$ with $\omega_Y^{\delta}\simeq \sO_Y$, for some $\delta\in \bN$, 
is known by the combination of 
\cite[Thm.~6]{Vie10} and \cite[\S6]{VZ02}. 

Generalization of moduli properties of 
canonically polarized varieties to those with no canonical choice of polarization 
is rarely straightforward. However, recently there has been significant progress in 
bridging this gap, especially in the context of hyperbolicity, including \cite{Deng22}
and \cite{JLSZ}.

\subsubsection{Further applications.}
We can apply the above strategy to establish further boundedness 
results for a much larger class of families of varieties; 
those with a good minimal model (\autoref{def:gmm}). 
As usual for such families one would expect 
boundedness in terms of birational classes (in the sense of \autoref{def:BB})
rather than isomorphic ones.
\begin{theorem}
\label{thm:main}
Let $V$ be any smooth quasi-projective variety and $h\in \bQ[x]$. 
The set of smooth projective families of polarized varieties over $V$ with a fixed Hilbert polynomial, 
a good minimal model and of maximal variation is birationally bounded. 
\end{theorem}

 \subsubsection{Absence of a logarithmic reformulation.}
 It is  worth pointing out that boundedness results 
as in \autoref{thm:curve} cannot be readily obtained by reducing the problem 
to that of log-stable families (see \cite[\S.~8]{ModBook} for definitions and background). 
That is, if we endow a given family $f: U\to V$ corresponding to $\frM_{\overline h}^{\sa}$ 
(or any similarly-defined moduli of polarized varieties) with an auxiliary divisor $\Delta$ such 
that $f_U: (U, \Delta) \to V$ is log-stable, we cannot use a logarithmic analogue of 
\cite{KoL10} to conclude \autoref{thm:curve}. 
The reason is that one first 
needs to prove a boundedness result similar to \autoref{CB} for moduli maps of such log-stable families
through Arakelov-type inequalities (see \ref{Arakelov}). 
However the bound in such inequalities will necessarily encode information about the 
locus over which $(U, \Delta)$ fails to be relatively log-smooth (see \cite{KT21} for 
another context where this problem appears). Failure to establish \autoref{CB} 
following this strategy 
stems from the fact that a priori there is no uniform control over this locus, which 
can be arbitrary large as one considers all families in $\frM_{\overline h}^{\sa}$. 
In fact, an effective bound can be found for this relatively non-log-smooth locus 
as a consequence of \autoref{thm:curve}.  

Currently we are not aware of a formulation of \autoref{CB} where fibers
are assumed to be of an appropriately defined singular class. 
See \cite{KT21} and \cite{KT25} for related problems. 
  
\subsection{Structure of the paper.}
After introducing preliminary notions and background in \autoref{Section2-Prelim}, 
we define and construct all relevant notions of moduli functors in 
\ref{SS:functors}. The Hilbert functor of doubly embedded schemes 
is constructed in \autoref{ss:Hilb2}. The fact that this functor is representable plays an important 
role in our approximation strategy, 
through its connection to compatibility properties with variation (\autoref{def:comp}). This is the content of \autoref{ss:good}.
In \autoref{WBGeometry} we introduce various notions of boundedness 
for morphisms and more generally rational maps, which naturally 
arise in our approach. In \autoref{Section3-WB} we apply the constructions and 
results in earlier sections to bound morphisms to the parametrizing space 
of the Hilbert functor of doubly embedded schemes. The crucial technical input here is 
\autoref{lem1}. \autoref{Section5-Bounded} contains the proof 
of our main result; \autoref{main:maxvar}. The proof uses, among 
other things, a technical lemma contained in \cite{Taji20} for which 
we provide a self-contained exposition 
in the appendix (\autoref{Section5-Appendix}).

\subsection{Acknowledgements}
Both authors owe a special thanks to S\'andor Kov\'acs for fruitful discussions.
We would also like to thank Dori Bejleri, Jack Hall, J\'anos Koll\'ar, Max Lieblich, and Siddharth Mathur for answering our questions and helpful comments.
KA is especially grateful to the Sydney Mathematical Research Institute (SMRI) for their financial
support and excellent working conditions during his visit. KA was partially supported by NSF grants DMS-2140781 and DMS-2302550, the SMRI International Visitor Program, and a UCI Chancellor's Fellowship.  BT was partially supported by Australian Research Council Discovery 
Project grant DP240101934 ``Deformation 
of Singularities through Hodge Theory and Derived Categories".

\section{Moduli functors, approximation and variation}
\label{Section2-Prelim}

In this article a variety is an irreducible, reduced and finite type scheme over $\mathbb C$.  
For the basic definitions and properties of singularities, including those of \emph{semi-log-canonical singularities}, or \emph{slc} for short, 
we refer the reader to \cite{KollarSingsOfTheMMP}. 
We will generally follow \cite{ModBook} for notions
of stability of varieties and their 
moduli problems.

\subsection{Stability in flat families of varieties}

\begin{notation}
For a finite type morphism $f:{X\to B}$ of schemes $X$ and $B$, we define 
$\omega_{X/B} : = \mathcal H^{-n} (f^{!}\mathscr O_B)$, where $n=\dim (X/B)$.
When $X$ is $S_2$ and $G_1$ (or normal) and $B$ is Gorenstein (or smooth), we have $\omega_{X/B} \simeq \sO_X(K_X - f^*K_B)$, 
with $K_X$ and $K_B$ being the canonical 
divisors, cf.~\cite[Def.~1.3]{KollarSingsOfTheMMP}.
\end{notation}

\begin{notation}[Reflexive hulls]
Given a reduced scheme $X$ that is $S_2$ and a coherent sheaf $\sF$ locally free in codimension one, for any
  $m\in \bN$, we define $\sF^{[m]} : = \big( \sF^{\otimes m} \big)^{**}$ and $\bigwedge^{[m]}\sF:= ( \bigwedge^m \sF)^{**}$, where
  $(\sblank)^{**}$ denotes the double dual. We also set $\det (\sF) := \bigwedge^{[r]}\sF$, $r= \rank(\sF)$.
\end{notation}

\begin{notation}[Pullback]
\label{Pull}
For any morphisms of schemes $f: X\to B$ and $\psi: B' \to B$, $X_{B'}$ denotes $X \times_{B} B'$. 
Let $\psi': X_{B'}\to B'$ and $f': X'\to B'$ denote the naturally induced projections. 
Given an $\sO_X$-module sheaf $\sF$ on $X$, we use 
$\sF_{B'}$ to denote $(\psi')^*\sF$.
For a pair $(X, \sF)$, we sometimes use $(X, \sF)_B$ to denote 
$(X_{B'}, \sF_{B'})$.
\end{notation}

When $f:X\to B$ is a flat and projective morphism of schemes with connected, positive-dimensional geometric fibers, we 
say $f$ is a flat projective family or just a \emph{family}, to be concise.

\begin{definition}[Locally-stable flat families]
\label{def:LS}
Let  $f: X \to B$ be a flat projective family over a reduced scheme $B$. 
Assume that $\omega_{X/B}$ is locally free in codimension one and $X$ is $S_2$.
We say that $f$ is \emph{locally-stable}, 
if 
\begin{enumerate}
\item $\omega^{[m_0]}_{X/B}$ is invertible for some $m_0\in \bN$, and 
\item $X_b$ is slc, for every closed point $b \in B$. 
\end{enumerate}
We refer to $m_0$ as an \emph{index} for the family. 
\end{definition}

\begin{remark}[Base change]
\label{rem:BC}
Let $f:X\to B$ be a locally-stable family. 
According to \cite[Thm.~1]{Kol18} 
if the general fiber is normal, then, 
for every $m\in \bZ$, $\omega^{[m]}_{X/B}$ is flat over $B$, 
and for any morphism $\psi: B'\to B$, the isomorphism 
$(\omega^{[m]}_{X/B})_{B'}\simeq \omega^{[m]}_{X'/B'}$ holds 
and that 
\[ 
\big( f' \big)_* \omega^{[m]}_{X'/B'} \simeq \psi^* f_* \omega^{[m]}_{X/B} ,
\]
with $f'$ and $X'$ being as in \autoref{Pull}. 
\end{remark}

\begin{definition}[Locally-stable reduction]
\label{KolRed}
Let $f^0:U\to V$ be a flat projective family, with $U$ and $V$ being quasi-projective. 
We say $f^0$ has a \emph{locally-stable reduction}, if $f^0$ has an extension $f:X\to B$, where $X$, $B$ and $f$
are projective, $V\subseteq B$ and $U\subseteq X$ are open subsets, up to isomorphism, and with $X_V= U$, together 
with the following property: 
there is 
 a projective variety $B'$ with
 \begin{itemize}
 \item a generically finite, surjective morphism $g: B'\to B$,
 \item  a locally-stable family $f':X'\to B'$,
 \end{itemize} such that  $X'_{V'} \simeq X\times_V V'$, where  $V':= g^{-1}(V)$.
\end{definition}

\subsection{Relevant moduli functors of polarized schemes and stacks}
\label{SS:functors}
For a detailed background on moduli problems of projective varieties 
we refer to \cite{Viehweg95}. An extensive treatment of the theory of stacks 
can be found in \cite{LMB00} and \cite{Olsson16}. 

The moduli problem $\frM$ of polarized schemes (\cite[\S.1.1]{Viehweg95}) consists of \emph{objects}
$$
\frM({\bC}) = \{  [(Y,L)]  \; \big| \; \text{$Y$ is a projective reduced scheme and $L$ an ample line bundle}\} ,
$$
where $[(Y,L)]$ denotes the isomorphism class of the pair $(Y,L)$, and a functor $\frM : \mathfrak{Sch}_{\bC} \to 
\mathfrak{Sets}$ defined by 
\begin{equation*}
\begin{aligned}
\frM (V)  = &  \Big\{ \text{Pairs} \; (f: U\to V, \sL)\;  \bigm\mid  \text{$f$ is flat and projective, $\sL$ is invertible} \\ 
                   & \;\;  \;\;  \;\;  \;\;  \;\;  \;\;   \;\;  \;\;  \;\;  \;\;  \;\;  \;\;  \;\;  \;\;   \;\;  \;\;  \;\;  \left. \text{and $(U_v,\sL_v)\in \frM_{\bC}$, for all $v\in V$} \Big\} \right/ \sim ,
\end{aligned}
\end{equation*}
where the equivalence relation $\sim$ is given by
\begin{enumerate} 
\item\label{1}
$$
\begin{aligned}
\big(  f_1: U_1 \to V , \sL_1  \big)  \sim  \big( f_2: U_2 \to V , \sL_2   \big)   \iff & \text{there is a $V$-isomorphism $\sigma: U_1\to U_2$}  \\
        & \text{such that $ \sL_1 \simeq \sigma^* \sL_2 \otimes f_1^*\sB$}, \\
        & \text{for some line bundle $\sB$ on $V$}, or
\end{aligned}
$$
\item\label{2} there is $\sigma$ as in \autoref{1}, with $\sB =\sO_V$, and $\sL_1 \equiv_V \sigma^* \sL_2$, or 
\item $\sim$ is trivial, i.e., is defined by identity.
\end{enumerate}
We sometimes refer to the objects of this moduli 
problem using $\mathfrak M({\bC} )= \mathfrak M(\Spec \mathbb C) = \Ob(\frM)$.

We say a subfunctor $\frM' \subseteq \frM$ is a \emph{moduli functor}, 
if it is closed under arbitrary base-change, cf.~\cite[Def.~1.3]{Viehweg95}. 

\begin{terminology}
We refer to $\sL$ in $(f: U \to V, \sL) \in \mathfrak M(V)$ as a \emph{polarization} for $f$. 
\end{terminology}

For the notion of coarse moduli space associated to a given moduli functor (when it exists), 
we refer to \cite[Def.~1.10]{Viehweg95}.

\begin{definition}[\protect{\cite[Def.~2.2]{Hassett-Kovacs04}}]
\label{def:openclosed}
Let $\mathfrak F \subset \mathfrak M$ be a moduli subfunctor. We say $\mathfrak F$ is \emph{open}, if 
for every $(f: X \to B, \sL) \in \mathfrak M(B)$ the set 
$V=\{ b \in B \; |\;  (X_b, \sL_b) \in \Ob(\mathfrak F)  \}$ is open in $B$ and $(X_V \to V, \sL_V) \in \mathfrak F(V)$.
The moduli subfunctor $\mathfrak F\subset \mathfrak M$ is \emph{locally closed}, if   
for every $(f: X \to B, \sL) \in \mathfrak M(B)$, there is a (universal) locally closed subscheme $j: B^u \hookrightarrow B$
such that for every morphism $\phi: T\to B$ we have: $(X_T \to T, \sL_T) \in \mathfrak F(T)$, if and only if 
there is a factorization
$$
\xymatrix{
T \ar@/^5mm/[rr]^{\phi} \ar[r] & B^u \ar[r]^{j}  & B. 
}
$$
\end{definition}

We note that by definition $\mathfrak F \subset \mathfrak M$ is open, if and only if it is locally closed 
and $B^u$ is open.

\begin{notation}
Given a moduli functor of polarized schemes $\mathfrak F$, by $\mathfrak F_h$ 
we denote the moduli subfunctor whose objects $(Y,L)$ have $h$ as their Hilbert polynomial with respect to 
$L$. 
\end{notation}

\begin{definition}[\protect{\cite[Def.~1.15.(1)]{Viehweg95}}]
\label{def:bounded}
We say a moduli subfunctor $\mathfrak F_h \subset \mathfrak M$ is \emph{bounded}, if there is $a_0\in \bN$ such that, 
for every $(Y,L) \in \Ob (\mathfrak F_h)$ and any $a\geq a_0$, the line bundle $L^a$ 
is very ample and $H^i(Y, L^a)=0$, for all $i>0$.
\end{definition} 

\begin{remark}\label{rk:bounded}
By Matsusaka \cite{Mat72} every $\mathfrak M'_h \subset \mathfrak M$ where $\Ob(\mathfrak M'_h)$ only consists 
of smooth (polarized) varieties is bounded. 
\end{remark}


\begin{remark}[Separated moduli functors]
For the definition of a separated functor of polarized schemes 
we refer to \cite[Def.~1.15.(2)]{Viehweg95}.
When $\frM'\subset \frM$ is separated, locally-closed and bounded, and its objects have finite (or more generally proper) automorphisms, 
thanks to Keel--Mori \cite{KeelMori}, there is an algebraic coarse moduli space  of finite type 
associated to $\frM'$ (see \cite{Knu71} for a comprehensive treatment of algebraic spaces).
In particular its associated stack $\mathcal M'$ is a separated Deligne--Mumford stack of finite type. 
\end{remark}

\subsubsection{Moduli functors for non-canonically polarized smooth varieties}\label{sec:modulifunctors}
Following \cite[p.~12]{Viehweg95} from now on we will assume that all base schemes for moduli
functors are separated and of finite type. 

\begin{def-thm}[\protect{Viehweg's moduli of polarized manifolds with 
semi-ample canonical divisor, \cite[Thm.~1.13]{Viehweg95}}]
\label{def:sa}
Let $\overline h\in \bQ[x_1, x_2]$ with $\overline h(\bZ \times \bZ)\subseteq \bZ$.
The objects of the subfunctor $\frM_{\overline h}^{\sa} \subset \frM$ consist of 
polarized smooth projective varieties $(Y,L)$ with semi-ample $\omega_Y$ 
satisfying  $\overline h(\alpha, \beta)  = \chi(\omega_Y^{\alpha} \otimes L^{\beta})$, 
for any $\alpha, \beta\in \bN$. The relation $\sim$ is defined as in \autoref{1}. 
This is sometimes referred to as the \emph{doubly-polarized moduli functor of 
manifolds with semi-ample canonical bundle}. 
The functor $\mathfrak M^{\sa}_{\overline h}$
has a quasi-projective coarse moduli space $M^{\sa}_{\overline h}$. 
We denote the corresponding stack by $\mathcal M_{\overline h}^{\sa}$. 
\end{def-thm}

\begin{remark}\label{rem:P}
The moduli functor $\mathfrak P_{\h}^{\sa}$ in the introduction is defined by $\mathfrak M_{\h}^{\sa}/\equiv$, 
where $\equiv$ is the numerical equivalence relation \autoref{2}. 
In particular there is a natural map $\mathfrak M_{\h}^{\sa}\to \mathfrak P_{\h}^{\sa}$. 
As explained in \ref{background-SemiAmple}, 
we will not make use of $\mathfrak P_{\h}^{\sa}$.
\end{remark}

\begin{assumption}
In the current paper, with no loss of generality we will restrict ourselves to moduli functors over reduced base schemes.
\end{assumption}

\begin{definition}\label{def:gmm}
Given two quasi-projective varieties $U'$ and $V$, with $V$ being smooth, we say 
$U'$ is a \emph{family of good minimal models over $V$}, 
if there is a flat projective family 
$f': U'\to V$ and an integer $N\in \mathbb N$ such that, 
for every $v\in V$, the geometric fiber $U'_v$ has 
only canonical singularities 
and that the reflexive sheaf 
$\omega_{U'/V}^{[N]}$ is invertible and $f'$-semi-ample
(i.e. the natural morphism $(f')^*(f')_*\omega_{U'/V}^{[mN]} \to \omega_{U'/V}^{[mN]}$ 
is surjective for some $m\in \bN$). 
We sometimes refer to $f'$ as a \emph{relative good minimal model}.
When $V=\Spec \bC$, this coincides with the standard notion of a good minimal model
variety. 
A smooth projective family $f: U\to V$ is then said to have a good minimal model over $V$, 
if $U$ is birational to a good minimal model $U'$ over $V$.
\end{definition}

\begin{defnot}\label{def:gmmMod}
We denote the moduli functor of polarized smooth projective varieties with good minimal models 
with Hilbert polynomial $h$ by $\mathfrak M_h^{\gmin}$. The relation $\sim$ is defined by \autoref{1}.
Furthermore, we use the notation $\frM_h^{\sa}$, when each $(Y,L)\in \Ob(\frM_h^{\sa})$ has semi-ample canonical bundle, 
with $Y$ being smooth.
\end{defnot}
\begin{remark}
\label{rk:lc}
$\frM_h^{gmin}$ is an open \cite[Thm.1.2, Cor. 1.4]{HMX18} (and therefore locally closed), separated \cite{Mat-Mum64} and 
bounded (\autoref{rk:bounded}) subfunctor of $\frM$.
\end{remark}
\begin{remark}[Comparing $\frM_{\h}^{\sa}$ and $\frM_h^{\sa}$]
\label{rk:compare}
With $\overline h \in \bQ[x_1, x_2]$ as in \autoref{def:sa}, and $\sim$ as in \autoref{1}, by definition we have 
$$
\big(  f: X\to B, \sL  \big) \in \frM_{\overline h}^{\sa}  \implies  \big( f: X\to B  , \omega_{X/B}\otimes \sL\big) \in \frM_h^{\sa}  ,
$$
where $h(x) : = \overline h(x, x)$. 
\end{remark}

\subsubsection{A bounded moduli functor for families of good minimal models}
For certain bounded moduli functors of good minimal models  
one can construct an algebraic coarse space of finite type.
\begin{def-thm}[\protect{\cite[\S3]{Taji20} and references therein}]\label{def:boundedgmm}
For a positive integer $N$, we define the \emph{bounded} moduli subfunctor $\frM^{[N]}\subset \frM$ by: 
$(Y,L) \in \Ob(\frM^{[N]})$ if, 
\begin{enumerate}
\item \label{object1} $Y$ has only canonical singularities, 
\item \label{object2} $\omega_Y^{[N]}$ is invertible and semi-ample ($N$ is not necessarily the minimum such integer), and 
\item \label{object3} for all $a\geq 1$, the line bundle $L^a$ is very ample with $H^i(Y,L^a)=0$, for all $i>0$.
\end{enumerate}
The relation $\sim$ is defined by \autoref{1}. For any $h\in \bQ[x]$, the moduli functor $\frM^{[N]}_h$ is separated and 
has an algebraic space of finite type $M_h^{[N]}$ as its coarse moduli space. 
\end{def-thm}

\begin{remark}[\protect{\cite[Lect.~6]{CKM88}, \cite[Claim.~3.6]{Taji20}}]
Assuming that $V$ is smooth, for any $(f_U: U\to V) \in \frM^{[N]}_h(V)$, the  
integer $N$ in \autoref{def:boundedgmm} is 
an index for $f_U$ (see \autoref{def:LS}).
\end{remark}


\begin{remark}[\protect{\cite[Thm.~3.13]{Taji20}}]\label{rk:LineExists}
  For any family  $f: U\to V$ of good minimal models of index $N$ there are a line bundle $\sL$ 
  on $U$ (not unique)
  and $h\in \bQ[x]$ such that $(f: U \to V, \sL) \in \frM^{[N]}_{h}(V)$.
  \end{remark}

\subsubsection{Isomorphisms of polarized families after a base change}
We will need the following well-known fact about families of polarized varieties with finite automorphism groups, 
cf.~\cite[7.5]{PZ19} and \cite[Lem.~7.4]{KK10}. 
\begin{fact}\label{fact:LiftIsom}
Let $\mathfrak F\subset \frM$ be a separated moduli subfunctor,  
 $V$ any 
quasi-projective variety
and $(U_i, \sL_i)\to V$, $i=1,2$, two polarized flat families over $V$. 
Assume that, for $i=1,2$ and every $v\in V$, we have 
$| \Aut(U_i, \sL_i)_v |< \infty$. 
If for each $v\in V$ there is an isomorphism of polarized schemes $\psi_v: (U_1, \sL_1)_v \to (U_2, \sL_2)_v$, 
then there are a surjective, finite morphism $\sigma: V'\to V$, where $V'$ is a quasi-projective variety, 
and an isomorphism $(U_1, \sL_1)_{V'} \simeq_{V'}  (U_2, \sL_2)_{V'}$, that is compatible 
with each $\psi_v$, at a general closed point $v\in V$.
\end{fact}

\subsubsection{The Hilbert functor of strictly embedded schemes}
\label{ss:Hilb}
Let $\frM'_h\subset \frM$ be a bounded subfunctor.
By definition, there is $k\in \bN$ such that,
for every $(Y,L)\in \Ob(\frM'_h)$, the line bundle $L^k$ 
is very ample and $H^i(L^k)=0$, for every $i>0$. 
As a consequence, there exists $d\in \bN$ such that, for every $(Y,L)\in \Ob(\frM'_h)$, 
\begin{enumerate}
\item \label{FunProps1} the natural map $\phi_{|L^k|}: Y \to \bP^d$ is an embedding. 
\end{enumerate}
Assume further that $\frM'_h$ is locally closed, and define the \emph{Hilbert functor 
$\mathfrak H^{d,k}_{\frM'_{h}} : \mathfrak{Sch}_{\bC} \to \mathfrak{Sets}$ of strictly 
embedded schemes in $\frM'_h$} by: 
\begin{equation*}
\begin{aligned}
\mathfrak H^{d,k}_{\frM_h'}(B)  = &  \Big\{ (f: X\to B, \sL)\in \frM'_h(B)  \bigm\mid  \text{there is an isomorphism 
                 $i:\bP_B(f_*\sL^k)  \overset{\simeq\;}{\to}   \bP^d\times B$}   \Big\} 
\end{aligned}
\end{equation*}
so that, for the natural map $\psi_{|\sL^k/B|}: X\to \bP_B(f_*\sL^k)$, the pullback of the
composition $(i\circ\psi)_{X_b} : X_b \to \bP^d$ satisfies \autoref{FunProps1}, for each $b\in B$. 
In particular the function $b\mapsto h^0(\sL^k_b)$ is constant over $B$, equal to $d+1$.

\begin{observation}
We note that $\mathfrak H^{d,k}_{\frM'_h}$ is indeed a moduli functor:
Consider any morphism of reduced schemes $\pi: B'\to B$,
$(f: X\to B, \sL) \in \mathfrak H^{d,k}_{\frM'_h}$ and the naturally induced morphisms $\pi': X_{B'} \to X$ and $f': X_{B'}\to X$.
As $\sL^k$ is flat over $B$~\cite[Thm.~III.9.9]{Ha77}, $f_*\sL^k$ is locally free and that $R^if_*\sL^k=0$, for all $i>0$, cf.~\cite[Cor.12.9]{Ha77}.
Therefore, using (derived) base-change $\mathbf{L}\pi^* \mathbf{R}f_*\sL^k \simeq \mathbf{R}f'_*(\pi')^* \sL^k$, we have 
$\pi^*f_*\sL^k \simeq f'_* \pi^* \sL^k$. Thus, using functoriality of $\bP_B(\cdot)$, we then find the commutative diagram:
$$
\xymatrix{
& \bP_{B'} \big( f'_* (\pi^*\sL^k)  \big) \ar[ddl] \ar[rr]  &&  \bP_B(f_*\sL^k) \simeq \bP^d\times B \ar[ddr]  \\
& X_{B'} \ar[dl]_{f'}  \ar[rr]^{\pi'}   \ar[u]  &&  X \ar[dr]^{f}  \ar[u]    \\
B' \ar[rrrr]^{\pi} &&&&  B  .
}
$$
In particular, by the universal property of fiber products, there is a surjective morphism $\bP_{B'} (f'_*\pi^*\sL^k) \to  \bP^d\times B'$ 
implying that $\bP_{B'} (f'_*\pi^*\sL^k) \simeq  \bP^d\times B'$, 
 i.e. we have $\big( f', (\pi')^*\sL \big) \in \mathfrak H^{d,k}_{\frM'_{h}}(B')$.
\end{observation}  

Now, let $(f^0:U\to V,\sL)\in \frM'_h(V)$ and set $V^0\subseteq V$ to be 
any open subset over which $(f^0_*\sL^k)|_{V^0}$ is free.
Denote the restriction of $f$ to $V^0$ by $f': U_{V^0}\to V^0$. 
By the assumption on $V^0$ we have $\bigoplus_{i=1}^{d+1} \sO_{V^0} \simeq f_*'\sL^k_{U_{V^0}}$.
After pulling back by $f'$ we find 
$$
\bigoplus_{i=1}^{d+1} \sO_{U_{V^0}} \simeq (f')^* f_*'\sL^k_{U_{V^0}} \longrightarrow  \sL^k_{U_{V^0}}.
$$
With the latter morphism being a surjection (by our assumption on $k$), we find the naturally induced map
\begin{equation}\label{embed}
U_{V^0} \hooklongrightarrow \bP_{V^0} ( f'_*\sL^k_{U_{V^0}} )  \simeq  \bP^d\times V^0, 
\end{equation}
i.e., the diagram 
$$
\xymatrix{
U_{V^0} \ar[d]  \ar@{^{(}->}[rr]^{\psi_{|\sL^k/V^0|}}  &&  \bP^d\times V^0 \ar[d]^{\pr_2}  \ar[r]^(.6){\pr_1} &  \bP^d \\  
V^0      \ar[rr]^=     &&    V^0
}
$$
commutes and $\sL^k \simeq_{V^\circ} \big(  \pr_1\circ \psi_{|\sL^k/V^0|} \big)^* \sO_{\bP^d} (1)$.

Next, using the assumption that $\frM'_h$ 
is locally closed, let $H^{d,k}_h\subseteq \Hilb^d_h$ be the universal locally 
closed subscheme for $\mathfrak H^{d,k}_{\frM'_h}$
and $f_{\X}: \X \subseteq \bP^d\times H^{d,k}_h \to H^{d,k}_h$  the universal 
object of $\mathfrak H^{d,k}_{\frM'_h}$, naturally defined by 
the restriction of the universal family over $\Hilb^d_h$ to $H^{d,k}_h$.
As such, for every $f^0$ and $V^0$ as above, there is a morphism $\phi_{V^0}: V^0\to H^{d,k}_h$ 
(unique up to isomorphism) such that 
\begin{equation}\label{eq:pb}
\Big( \X , \sL_{\X} \Big)_{V^\circ}  \simeq_{V^\circ} \big(U_{V^0} , \sL^k\big), 
\end{equation}
where $\sL_{\X}$ is defined by the restriction of $\pr_1^*(\sO_{\bP^d}(1))$, 
with $\pr_1: \bP^d \times H^{d,k}_h \to \bP^d$, to the universal family $\X$
(see \autoref{Pull}).

\subsubsection{The Hilbert functor of doubly embedded schemes} 
\label{ss:Hilb2}
Generalizing the constructions in \autoref{ss:Hilb} we can consider a moduli functor whose objects consist of objects of a 
bounded moduli functor that are equipped with multiple embeddings arising from various choices of 
exponents for the given polarization. In the current paper it suffices to consider the case of 
two embeddings. 

With $\frM'_h$ and $k,d\in\bN$ as in \autoref{ss:Hilb}, there is $l\in \bN$ such that, 
for every $(Y,L)\in \Ob(\frM'_h)$, the map $\phi_{|L^{k+1}|}: Y\to \bP^l$ is an embedding.
We thus similarly define 
\begin{equation*}
\begin{aligned}
\mathfrak H^{d,l,k}_{\frM_h'}(B)  = &  \Big\{ (f: X\to B, \sL)\in \frM'_h(B)  \bigm\mid  \text{there is an isomorphism}\\
               & \;\;\;\;\;\;\;\;\;\;\;\;\; \;\;\;\;\;\;\;\;\;\;\;\;  \;\;\;\;\;\;\;\;\;\;\;\;   \bP_B(f_*\sL^k) \times_B \bP(f_*\sL^{k+1}) \overset{\simeq\;}{\longrightarrow}   \bP^d\ \times \bP^l \times B  \Big\} .
\end{aligned}
\end{equation*}

For every $(f^\circ: U\to V,\sL) \in \frM'_h(V)$, we similarly set 
$V^\circ \subseteq V$ to be any open subset over which $f^\circ_*\sL^k$ and $f^\circ_*\sL^{k+1}$ 
are both free. Using the same notations 
as in \autoref{ss:Hilb}, similarly to \autoref{embed} and using the universal property
of fiber products, we thus find  
the naturally induced embedding over $V^\circ$:
$$
U_{V^\circ} \overset{\psi}{\hooklongrightarrow} \bP_{V^\circ}(f'_*\sL^k) \times_{V^\circ} \bP_{V^\circ}( f'_*\sL^{k+1} ) \cong \bP^d\times \bP^l \times V^\circ 
$$
such that $\sL^k \simeq_{V^\circ} \psi^*\big( \pr_1^* \sO_{\bP^d} (1) \big)$ and 
$\sL^{k+1} \simeq_{V^\circ} \psi^*\big( \pr_2^*\sO_{\bP^l} (1) \big)$. 
In particular we have 
\begin{equation}\label{eq:L}
\sL \simeq_{V^0} \psi^*\big( \pr_1^*\sO_{\bP^d} (-1)  \otimes \pr_2^* \sO_{\bP^l} (1)  \big)  .
\end{equation}

Again, similarly to \autoref{ss:Hilb}, and noting that the Hilbert functor of closed subschemes of 
$\bP^d\times \bP^l$ (with fixed Hilbert polynomial $h$) is represented by a scheme $\Hilb^{d,l}$
(resp. quasi-projective scheme $\Hilb^{d,l}_h$), we find 
\begin{enumerate}
\item the universal locally closed subscheme $H^{d,l,k}\subseteq \Hilb^{d,l}_h$ for 
$\mathfrak H^{d,l,k}_{\frak M'_h}$, and 
\item the associated universal family $f_{\X}: \X \overset{\eta}{\hooklongrightarrow}\bP^d \times \bP^l \times H_h^{d,l.k} \longrightarrow H_h^{d,l,k}$.
\end{enumerate}

Now, as in \autoref{ss:Hilb}, for $f^0: U\to V$ as above there are $\phi_{V^0}: V^\circ\to H_h^{d,l,k}$ and 
the isomorphism 
\begin{equation}\label{eq:pb2}
\X\times_{H_h^{d,l,k}} V^\circ \simeq_{V^\circ} U_{V^\circ}  .
\end{equation}
Furthermore, after replacing $H_h^{d,l,k}$ by an open subset if necessary,  $f_{\X}$ is polarized with 
$\sL_{\X}:= \eta^*\big( \pr_1^* \sO_{\bP^d} (-1) \otimes \pr_2^*\sO_{\bP^l}(1) \big)$, 
with respect to which, by \autoref{eq:L}, isomorphism \autoref{eq:pb2} holds 
as polarized families: 
\begin{equation}\label{isom:pb2}
\Big( \X  , \sL_{\X} \Big)_{V^\circ}  \simeq_{V^\circ} \big(U , \sL\big)_{V^0}.
\end{equation}

The above Hilbert functors are examples of \emph{representable} moduli functors. When a moduli functor $\mathfrak H$ of polarized schemes is representable, it is 
equipped with a universal family $(f_{\X}:  \X\to H, \sL_{\X})$ over a quasi-projective scheme $H$, 
i.e., for every $(f: U\to V, \sL)\in \mathfrak H(V)$, there is a morphism $\phi_V: V\to H$ 
of finite type such that $(U, \sL) \simeq_V (\X, \sL_{\X})_V$.

\begin{terminology}\label{term:approx}
We say a moduli functor $\frM'$ of polarized schemes can be \emph{approximated} by 
a representable functor $\mathfrak H$, if for every $(f: U\to V, \sL)$ 
there is an open subset $V^\circ \subseteq V$ such that 
$(U, \sL)_{V^\circ}\in \mathfrak H(V^\circ)$.
\end{terminology}

Using \autoref{term:approx}, isomorphism \autoref{isom:pb2} has the following interpretation. 

\begin{proposition}\label{prop:rep}
Let $\frM'_h$ be any locally closed and bounded moduli functor. Then, for every suitable choice of 
$k,d,l\in \bN$ as above, $\frM'_h$ can be approximated by the Hilbert functor 
of doubly embedded schemes $\mathfrak H^{d,l,k}_{\frM'_h}$ in $\frM'_h$.
\end{proposition}

\begin{notation}\label{not:H}
To lighten the notations, we will use $H$ to denote the 
parametrizing scheme $H^{d,l,k}_h$ (or $H^{d,k}$) of 
$\mathfrak H^{d,l,k}_{\frM'_h}$ (respectively $\mathfrak H^{d,k}_{\frM'_h}$).
\end{notation}

\begin{terminology}\label{term}
We say a morphism $\phi: V^\circ \to H\subseteq \Hilb^{d,l}_h$ \emph{factors through $\mathfrak H^{d,l,k}_{\frM'_{h}}$}, 
if it is induced by a polarized family $(U_{V^\circ}\to V^\circ, \sL)$ 
satisfying the isomorphism 
$$
( U_{V^\circ}, \sL ) \simeq_{V^\circ}  (\X, \sL_{\X})_{V^\circ}   ,
$$
with the right-hand side being the pullback via $\phi: V^\circ \to H$ (see \autoref{Pull}).
\end{terminology}

\subsection{Arakelov inequalities for families of bounded varieties}
\label{Arakelov}
A key technical ingredient for bounding coarse moduli maps (see \autoref{def:WB})
is the so-called Arakelov inequalities. In \autoref{KovTaj} below we give an incarnation 
of this inequality that is most relevant for our purposes. 

\begin{definition}
Given sheaves of $\sO_B$-modules $\sV$ and $\sF$, a
\emph{$\sV$-valued system} means a splitting $\sF = \bigoplus \sF_i$ and a sheaf
homomorphism $\tau: \sF \to \sV \otimes \sF$ that is Griffiths-transversal.  If we
assume further that $\sV= \Omega^1_B(\log D)$, $\tau$ is integrable and $\sF$ is
locally-free, then $(\sF, \tau)$ is referred to as a \emph{logarithmic system
of Hodge bundles}.
\end{definition}

\begin{terminology}
Given a smooth projective morphism $f^\circ: U \to V$ of smooth quasi-projective varieties 
$U$ and $V$, we call a projective morphism 
$f: X\to B$ of smooth projective varieties $X$ and $B$ a \emph{smooth compactification}, if 
$B\backslash D \simeq V$, where $D$ is the divisorial part of the discriminant locus 
of $f$, is assumed to be of simple normal-crossing support, and $X\backslash f^{-1}D \simeq U$. 
\end{terminology}

\begin{proposition}\label{KovTaj}
Let $\frM'_h \subset \mathfrak M$ be any bounded and locally closed 
moduli subfunctor (with $\frM$ as defined in \ref{SS:functors})
with only smooth objects. 
For every $m\in \bN$ there is $b_m\in \bZ^{\geq 0} [x_1, x_2]$ such that
for every smooth quasi-projective curve $C^\circ$ and 
compactification $C$, with $d:= \deg(C\backslash C^\circ)$, 
and $( f^\circ: U\to C^\circ )\in \frM'_h(C^\circ)$, the inequality 
$\deg( f_*\omega^m_{X/C} ) \leq b_m(g(C),d)$ holds, 
for any smooth compactification $f:X\to C$, with $g(C)$
denoting the genus of $C$.
\end{proposition}

\begin{proof}
Following the constructions in \cite{Vie-Zuo01} and in particular \cite[Props.~2.22, ~5.1]{KT22}, 
for any such $f:X\to C$ and $m\in \bN$ there 
is a logarithmic system of Hodge bundles $(\sG= \bigoplus \sG_i , \theta)$, 
with $\sG_0\simeq \det\big( f_*\omega^m_{X/C} (-mD) \big)$, 
where $D:= C\backslash C^\circ$. 
By applying~\cite[p.~28]{KT22}, over one-dimensional base schemes, 
we find $\overline b\in \bZ^{\geq 0}[x_1, x_2, x_3, x_4]$ 
such that, for any $C^\circ, f^\circ$ and $f$ as above, we have 
\begin{equation}\label{eq:ARAK}
\deg\big( \det f_*\omega^m_{X/C} \big)  \leq  \overline b( m, r_m, g, d ) ,
\end{equation}
where $g:= g(C)$ and $r_m:=\rank (f_*\omega^m_{X/C})$. 
Furthermore, as $\frM'_h$ is bounded (\autoref{rk:bounded}), 
according to \autoref{ss:Hilb},
each $\Ob(\frM'_h)$ is isomorphic to a geometric fiber 
of the universal family $\X$ of $\mathfrak H^{d,k}_{\frM'_h}$ 
(for some $d,k\in \bN$). Moreover, over a suitable 
open subset $C'\subseteq C^\circ$, there is a morphism $\phi_{C'}: C'\to H$ (\autoref{not:H})
such that 
\begin{equation}\label{iso}
\X\times_H C' \simeq U_{C'} , 
\end{equation}
as in \autoref{eq:pb}. 
With no loss of generality we may assume that $H$ is reduced
and by induction on $\dim H$ we can assume that 
$\phi_C'(C')\not\subseteq H\backslash H_{\reg}$, with $H_{\reg}$ denoting 
the smooth locus of $H$. 
Now, by applying Siu's invariance of plurigenera~\cite{Siu98}
to $\X$ over $H_{\reg}$, denoted
by $X_{\reg}$, we find that $(f_{\X})_* \omega^m_{\X_{\reg}/ H_{\reg}}$ 
is locally free 
over each irreducible component. 
Set $t_m$ to be the maximum of the rank of each of these 
finite number of locally free sheaves. 

On the other hand, using \autoref{iso}, we have 
$$
\big(  \omega^m_{X_{\reg}/ H_{\reg}} \big)_{C''} \simeq  \omega^m_{U_{C''}/C''}. 
$$
This implies that, for every $m\in \bN$, $C^\circ$ and $f^\circ\in \frM'_h(C^\circ)$, 
the inequality $\rank(f_*\omega^m_{X/C}) \leq t_m$ holds. 
Therefore, for every $m\in \bN$, the quantity $b_m(g,d) := \overline b(m, t_m, g,d)$
forms an upper-bound for the right-hand side of \autoref{eq:ARAK}, 
that is we can take $b_m(x_1, x_2) = \overline b(m, t_m, x_1, x_2)$.
\end{proof}

\subsection{Variation and compatibility with the moduli}
\label{def:var}
Let $f: X\to B$ be a flat projective family. 
Following ~\cite{Viehweg83} and Kawamata~\cite[pp.~5--6]{Kawamata85}, 
we define $\Var(f)$ by the transcendence degree of a \emph{minimal closed 
field of definition} $K$ for $f$. We note that $K$ is the minimal (in terms of inclusion)
algebraically closed field in the algebraic closure $\overline{\bC(B)}$ of the function field 
$\bC(B)$ for which there is a $K$-variety $W$ such that 
$X\times_B \Spec\big( \overline{\bC(B)} \big)$ is birationally equivalent 
to $W\times_{\Spec(K)} \Spec(\overline{\bC(B)})$. 

By definition we have $0\leq \Var(f) \leq \dim B$. 
We say $f$ has \emph{maximal variation}, when $\Var(f)= \dim B$.

\begin{definition}[Compatibility with variation]
\label{def:comp}
Let $\mathfrak F\subset \frM$ be a moduli subfunctor with an algebraic coarse moduli space 
$M$. Given a smooth quasi-projective variety $B$, 
let $(f: X\to B, \sL)\in \mathfrak F(B)$, with the induced 
moduli morphism $\mu: B\to M$. 
We say that \emph{the formation of $M$ is compatible 
with $\Var(f)$}, or that \emph{$M$ computes $\Var(f)$}, if for every subvariety $T\subseteq B$ 
with the induced moduli map $\mu_T:=\mu|_T$,
we have 
$$
\Var(f_T)= \dim T \iff  \mu_T \text{\; is generically finite},  
$$
where $f_T$ denotes 
the pullback of $f$ over $T$.
\end{definition}

\begin{lemma}[Existence of compatible polarizations]
\label{lem:comp}
Let $f: U\to V$ be a family of good minimal models of index $N$ over 
$V$. There are 
\begin{enumerate}
\item an open subset $V^\circ \subseteq V$, an invertible sheaf $\sL^\circ$ on
$U^\circ:= U_{V^\circ}$, and 
\item $h''\in \bQ[x]$
\end{enumerate}
such that, with $f^\circ:=f|_{U^\circ}$, we have $(f^\circ: U^\circ\to V^\circ, \sL^\circ) \in \frM^{[N]}_{h''}(V^\circ)$ 
(for the definition of $\frM^{[N]}$, see \autoref{def:boundedgmm})
and that the formation of $M_{h''}^{[N]}$ is compatible with $\Var(f)$.
\end{lemma}

\begin{proof}
The proof is contained in \cite[\S3.2]{Taji20}. 
For reader's convenience 
we have reproduced the details in the appendix of this paper (\autoref{Section5-Appendix}). 
\end{proof}

\subsubsection{Compatibility for representable functors through change of polarization}
\label{ss:good}
\begin{assumption}[Irreduciblity]
\label{ass:irred}
In this section we will assume that for the morphism $f_{\X}:\X\to H$
the two schemes $\X$ and $H$ in \autoref{ss:Hilb2} and \autoref{not:H} are irreducible and reduced. 
\end{assumption}

\begin{assumption}\label{assump}
Let $\frM''_{h''}$ be a separated moduli functor with 
Hilbert polynomial $h''$ satisfying the following properties.
\begin{enumerate}
\item \label{finiteAut} $\big| \Aut(X',L') \big| < \infty$, for every $(X',L')\in \Ob(\frM_{h''}'')$.
\item \label{coarse} $\frM''_{h''}$ has an algebraic coarse moduli space $M''_{h''}$ of finite type
that is compatible with the 
variation of $f_{\X'}$ in the sense of \autoref{def:comp}.
\end{enumerate}
\end{assumption}

\begin{definition}[Distinguished moduli functors]
\label{def:good}
We say a representable moduli functor $\mathfrak H$ 
is \emph{distinguished} if for every smooth quasi-projective variety $V$, 
any $f\in \mathfrak H(V)$ has a locally-stable 
compactification (\autoref{KolRed}) and the universal family $f_{\X}: \X\to H$
is birational to a family $f_{\X'}: \X'\to H'$, over an open subset $H'\subseteq H$, 
such that $f_{\X'}$ can be equipped with a polarization $\sL'$ for which we have 
$\big(f_{\X'}, \sL'\big) \in \frM''_{h''}(H')$, where $\frM''_{h''}$ is a moduli functor 
satisfying the assumptions made in \autoref{assump}.
\end{definition}

\begin{observation}\label{obs:change}
Let $\mathfrak H$ be any representable moduli functor of varieties with 
good minimal models. Over an open smooth subset $H_1\subseteq H$, let $f_{\X'} : \X'\to H_1$ be a good
minimal model of index $N$ for the universal family $f_{\X}: \X\to H$ restricted to $H_1$. 
Then, by \autoref{lem:comp}, over a possibly smaller open subset $H_2\subseteq H_1$, 
$f_{\X'}$ 
can be equipped with a polarization $\sL'$ for which we have 
$\big((f_{\X'})_{H_2}, \sL'\big) \in \frM^{[N]}_{h''}(H_2)$. 
Furthermore $\frM^{[N]}_{h''}$ has an algebraic coarse moduli space $M^{[N]}_{h''}$ of finite type
satisfying \autoref{coarse}. 
\end{observation}

\begin{proposition}[Approximation by distinguished functors]
\label{claim:good}
For any $h\in \bQ[x]$, let $\frM'_h$ be a locally closed and bounded moduli
functor of smooth projective varieties admitting good minimal models
with Hilbert polynomial $h$. Then, there 
are $d,l,k\in \bN$ for which $\mathfrak H^{d,l,k}_{\frM'_h}$ 
is a distinguished moduli functor approximation of 
$\frM'_h$.
\end{proposition}

\begin{proof}
By \autoref{prop:rep} it suffices to show that $\mathfrak H^{d,l,k}_{\frM'_h}$ is distinguished. 
According to Abramovich \cite[Cor.~A.2]{Deng22} 
every $f\in \frM'_h(V)$ has a locally-stable reduction. 
Now, using $\frM^{[N]}_{h''}$ in \autoref{obs:change}, it remains to 
show that the objects of $\frM_{h''}^{[N]}$ 
satisfy \autoref{finiteAut}. The latter has been shown in 
\cite[Rk.~3.12]{Taji20}. 
\end{proof}

\begin{example}\label{ex:good}
The two functors $\frM_h^{\gmin}$ and $\frM_h^{\sa}$ defined in \autoref{def:gmmMod} 
are examples of functors satisfying the assumptions 
of \autoref{claim:good} (see \autoref{rk:lc}).
\end{example}

\begin{notation}[Compactification of the universal family of $\mathfrak H^{d,l,k}_{\frM'_h}$]
\label{not:compact}
With $\frM'_h$ as in \autoref{claim:good}, in the situation of \autoref{obs:change}
let $\overline H$ be a compactification of $H$
and $D_{\overline H}\subset \overline H$
a reduced divisor such that $\overline H\backslash D_{\overline H}\hookrightarrow H_2$.
We denote a smooth compactification of $f_{\X}$ over $\overline H\backslash D_{\overline H}$ by 
$f_{\overline \X}: \overline \X\to \overline H$. 
By the construction above we have a finite type morphism 
$\mu: H_2\to M_{h''}''$. We may assume $M''_{h''}$ is reduced. 
Using the fact that $M''_{h''}$ is of finite type \autoref{coarse}, after replacing $\overline H\backslash D_{\overline H}$ by a smaller open subset 
we will assume that $M''_{h''}$ is a quasi-projective variety. 
Let $\overline M_{h''}''$ denote a projective compactification of $M_{h''}''$.
By a slight abuse of notation, we will denote the pullback of 
$f_{\X'}$ to $\overline H\backslash D_{\overline H}$ by $f_{\X'}$.
\end{notation}
\section{Weak boundedness notions and geometric consequences}
\label{WBGeometry}

In this subsection we introduce 
notions of boundedness, starting from coarse set-theoretic ones 
and concluding with a key 
geometric notion of boundedness for certain rational maps in 
\autoref{FACT2}.

\begin{definition}[Set-theoretic boundedness]
\label{SWB}
Let $H$ be any scheme and $S\subseteq H$ a subset. We say $S$ is set-theoretically bounded inside $H$, if there 
is a finite type subscheme $W\subseteq H$ containing $S$ as a subset.
\end{definition}

\begin{set-up}\label{prelim}
Let $V$ and $M$ be any two quasi-projective varieties and $\overline M$ a proper compactification of $M$.
Assume that $\lambda_{\overline M}$ is an ample line bundle on $\overline M$.
Let $b_{\lambda_{\overline M}}\in \bZ^{\geq 0}[x_1, x_2]$. 
\end{set-up}

The following definition of a \emph{weakly-boundedness morphism} is the same as \cite[Def.~2.4]{KoL10} 
but with a few 
technical differences that reflect relative divergence of   
arguments and context in the current article.

\begin{definition}[Weakly-bounded morphisms]
\label{def:WB}
In the setting of \autoref{prelim}
we say that a given morphism $\phi: V\to M$ of finite type is \emph{weakly-bounded} 
with respect to $\lambda_{\overline M}$ and $b_{\lambda_{\overline M}}$, 
if it satisfies the following condition: For every smooth quasi-projective curve $C^0$ 
and any finite type morphism $C^0 \to V$, the 
naturally induced morphism $\phi_C: C\to \overline M$ from a smooth projective compactification $C$ 
of $C^\circ$ satisfies the inequality 
\begin{equation}\label{eq:WB}
\deg \big(  \phi_C^* \lambda_{\overline M} \big) \leq b_{\lambda_{\overline M}}(g(C),d) ,
\end{equation}
where $g(C)$ denotes the genus of $C$ and $d:=\deg(C\backslash C^0)$.
We sometimes refer to $b_{\lambda_{\overline M}}$ as a \emph{weak bound}.
\end{definition}

\begin{definition}[Weakly-bounded stacks]
Let $\mathcal M$ be a separated Deligne--Mumford stack of finite type with a quasi-projective 
coarse moduli space $M$. If every morphism $\mu_V: V\to M$ factoring through 
$\mathcal M$ is weakly-bounded, then we say $\mathcal M$ is weakly-bounded. 
\end{definition}

\begin{notation}\label{not:frakC}
Let $V$ be a smooth quasi-projective variety. 
Given a moduli functor $\mathfrak M'\subseteq \mathfrak M$, we use
$\mathfrak C_{V, \mathfrak M'}$ to denote the set of all polarized 
$(f: U\to V) \in \mathfrak M'(V)$.
\end{notation}

\begin{definition}[Coarsely-bounded families]\label{def:CB}
Let $\mathfrak M'\subseteq \mathfrak M$ be a moduli functor with an algebraic 
coarse moduli space $M$ of finite type. We say that a subset $\mathfrak C\subseteq \mathfrak C_{V, \mathfrak M'}$
is coarsely-bounded, if there 
is a scheme $W$ and a morphism $\Phi: W\times V\to M$, both of finite type, such that 
for every morphism $\mu_V: V\to M$ arising from an element of $\mathfrak C$ 
, there is a closed point $w\in W$ for which we have $\Phi_{ \{w\} \times V } = \mu_V$.
\end{definition}

\begin{definition}[Polarized boundedness]\label{PolB}
In the setting of \autoref{def:BB}, assuming that each member of $\mathfrak C_V$ 
is polarized, then we say $\mathfrak C_V$ is bounded, if there is a line bundle $\sL_{\mathcal Y}$
on $\mathcal Y$ such that for every $(f^0: U\to V, \sL) \in\mathfrak C_V$ we have 
$(\mathcal Y, \sL_{\mathcal Y})_{ \{ w \} \times V } \simeq_V (U, \sL)$.
\end{definition}

The following theorem is due to Kov\'acs--Lieblich \cite[Thm.~1.7]{KoL10} but for weakly-bounded stacks.
However, it is easy to see from the proof (\cite[p.~605]{KoL10}) that 
this important result can be formulated in terms of coarse boundedness. 

\begin{theorem}[\protect{\cite[Thm.~1.7]{KoL10}}]\label{thm:KL}
Let $\mathcal M'$ be any compactifiable stack (\autoref{def:compact}) associated to a moduli functor $\mathfrak M'$.
If $\mathfrak C\subseteq \mathfrak C_{V, \mathfrak M'}$ 
is coarsely-bounded over $V$, then $\mathfrak C$ is bounded. 
\end{theorem}

\begin{notation}\label{not:smooth}
Given a smooth quasi-projective variety $V$ as above, let $\overline V$ be a 
projective compactification. 
By $B$ we will denote a strong desingularization of $\overline V$, i.e.
for the desingularization map $\pi: B \to \overline V$ the restriction
 $\pi^{-1}|_V$ is an isomorphism. By an abuse of notation 
 we will also denote the image of $\pi^{-1}|_V$ by $V$. 
 We further define the reduced divisor $D: = B \backslash V$ and assume that it has simple normal 
 crossing support. 
\end{notation}

\begin{remark}[Indeterminacy and weak-boundedness via the Picard scheme]
Given $B$ and $V$ as above, a set of weakly-bounded morphisms 
from $V$ to $\overline M$ does not necessarily extend to a 
set of morphisms from $B$ to $\overline M$ (let alone weakly-bounded ones). 
For the Picard scheme $\Pic_B$
however, this extension issue for morphisms is not relevant, thanks to extensions 
of reflexive sheaves over closed subsets of $\codim_B \geq 2$. 
This extra flexibility allows us to in particular consider the notion 
of weak-boundedness for rational maps. 
\end{remark}


\begin{definition}[Weakly-bounded rational maps]
\label{def:ratWB}
In the setting of \autoref{def:WB}, we say a rational map $\phi: V\dashrightarrow \overline M$
is weakly-bounded if, for a given closed subset $\Gamma\subset V$, \autoref{eq:WB} holds for every $C^0\to V$ as in \autoref{def:WB}
whose image in $V$ is not contained in the union of the indeterminacy locus of $\phi$
and $\Gamma$. 
\end{definition}

\begin{notation}[Extensions of invertible sheaves]
\label{not:e}
Let $\phi: B \dashrightarrow W$ be a rational map
of projective schemes. Assume that $B$ is smooth. 
Given a line bundle $\sL$ on $W$, by $(\phi^*\sL)_e$ we denote  
the unique extension of $\phi^*\sL$, where it is defined, to an invertible 
sheaf on $B$.
\end{notation}

In the following proposition we generalize \cite[Prop.~2.6]{KoL10} to the setting of 
weakly-bounded rational maps. 

\begin{proposition}
\label{FACT2}
In the setting of \autoref{prelim}, \autoref{not:smooth} 
and \autoref{not:e}, let $S_V$ be the set of all weakly-bounded rational maps from $V$ to $\overline M$. 
Then, 
\begin{enumerate}
\item\label{item:kelid0} the set 
$$
S_{\Pic_B} : = \Big\{ \big[  (\phi^*\lambda)_e \big] \; \big| \; \phi\in S_V  \Big\} \subset \Pic_B
$$
is set-theoretically bounded by a subscheme $Y\subset \Pic_B$. 
Moreover, 
\item\label{item:kelid2} there are a reduced scheme $W$ of finite type
$W$ and a rational map $\Phi: W\times V \dashrightarrow \overline M$ 
such that, for every $\phi \in S_V$, there is a closed point $w\in W$ with 
$\Phi_{\{ w \} \times V } = \phi$.
\end{enumerate}
\end{proposition}

\begin{proof}
Set $n:= \dim B$ and let $A$ be a very ample divisor in $B$. 
According to \cite[Thm.~XIII.3.13(iii)]{SGA6} to prove 
\autoref{item:kelid0} it suffices to find $b\in \bR^{\geq 0}$ 
such that for every $\phi\in S_V$ we have: 
\begin{enumerate}
\item\label{upper} $(\phi^*\lambda_{\overline M})_e \cdot A^{n-1} \leq b$ .
\item\label{lower} $(\phi^*\lambda_{\overline M})^2_e \cdot A^{n-2}\geq 0$ . 
\end{enumerate}

We note that when $\dim B=1$, then this immediately follows from 
\autoref{def:WB}, by taking $b: = b_{\lambda_{\overline M}} (g(C),d)$.
Therefore, we may assume that $\dim B>1$.

To prove \autoref{upper} we use the following claim. 
\begin{claim}\label{claim:invariance}
Let $i: B \hookrightarrow \bP^m$ be an embedding with $A\simeq i^*\sO_{\bP^m}(1)$.
There is a finite set $\overline G$ such that 
for every smooth projective curve $C\subset B$, cut out by an
$(n-1)$-tuple  $(A_1, \ldots, A_{n-1})$ of members of 
$|\sO_{\bP^m}(1)|$ restricted to $B$, we have $g(C)\in \overline G$.
\end{claim}
\noindent
\emph{Proof of \autoref{claim:invariance}.} 
We use 
the universal family $\X^u$ of complete intersection curves in $\bP^m$
defined by members of $|\sO_{\bP^m}(1)|$:
$$
\xymatrix{
\X^u \ar[dr]   \ar@{^{(}->}[r]^(.2){\subset}   & \bP^m \times \prod^{n-1} \bP\big( H^0( \sO_{\bP^m}(1) ) \big)    \ar[d]   \\
            & \prod^{n-1} \bP\big( H^0( \sO_{\bP^m}(1) ) \big) .
}
$$
(Explicitly $\X^u$ is defined by the common zeros of $(n-1)$ equations $\sum_{i=0}^{m+1} s_i f_{ij}$, $j=1, \ldots, n-1$, where
 $\{ s_i\}_{i=0}^{m+1}$ is a basis for $H^0(\sO_{\bP^m}(1))$ and $\{ f_{ij}  \}_{i=0}^{m+1}$ is a set of 
 linear parameters for the $j$-th factor in $\prod^{n-1} \bP( H^0(\sO_{\bP^m}(1)) )$.) 
 We then restrict this family to $B$ and consider the induced morphism 
 $$
 F: \mathfrak Y^u:= \X^u \cap \Big( B \times \prod^{n-1} \bP\big( H^0( \sO_{\bP^m}(1) ) \big) \Big)  
     \longrightarrow \prod^{n-1} \bP\big( H^0( \sO_{\bP^m}(1) ) \big) .
 $$
The claim now follows from the invariance of genus over the smooth
locus of each irreducible component of $\mathfrak Y^u$.  \qed \\ 

Define $b:= \max_{g\in \overline G} b_{\lambda_{\overline M}}(g, d)$, 
where $d:= A^{n-1}\cdot D$. For every $\phi\in S_V$, 
let $C\subset B$ be any smooth projective curve as in \autoref{claim:invariance} and 
additionally satisfying the assumptions made in \autoref{def:ratWB}. 
As such, we have $\phi^*_C\lambda_{\overline M} = (\phi^*\lambda_{\overline M})_e\big|_C$, 
and thus find that 
$$
\big(\phi^* \lambda_{\overline M} \big)_e \cdot A^{n-1} 
\leq b_{\lambda_{\overline M}}(g(C),d) \underbrace{\leq}_{g(C)\in \overline G} b ,
$$
as required for \autoref{upper}.

For \autoref{lower}, let $\phi$ be any member of $S_V$. 
Take $T\subset B$ to be a smooth complete-intersection surface 
defined by sufficiently general members of the linear system $|A|$ so that 
the indeterminacy of $\phi|_T$ is of $\codim_T \geq 2$. 
For such $T$ we have 
$$
\big(\phi^* \lambda_{\overline M}  \big)_e  \Big|_T   \simeq \Big(  \big( \phi\big|_T \big)^* \lambda_{\overline M} \Big)_e  .
$$
In particular, we find 
\begin{equation}\label{eq:power}
\big(\phi^*\lambda_{\overline M} \big)^2_e  \cdot A^{n-2} =   \Big(   \big(  (\phi|_T )^*\lambda_{\overline M}  \big)_e  \Big)^2 .
\end{equation}
As $\big(  (\phi|_T)^*\lambda_{\overline M} \big)_e$ is globally generated in codimension one, the right-hand side 
of \autoref{eq:power} is non-negative, establishing \autoref{lower}, and therefore \autoref{item:kelid0}.

Item \autoref{item:kelid2} follows from \autoref{item:kelid0} together with \cite[Lem.~2.9, p.~593]{KoL10}.  
More precisely, we first assume that $\overline M = \bP^k$, 
$\lambda_{\overline M} =\sO_{\bP^k}(1)$. According to \cite[Lem.~2.9]{KoL10} 
the functor  
that assigns to $Y$ sets of tuples $(\sL; s_0,\ldots, s_k)$, 
where the numerical class $[\sL]$ belongs to $Y$ and $s_i\in H^0(B, \sL)$, 
at least one of which is non-zero, is representable by
$W'\times B$, where $W'$ is a reduced scheme of finite type. 
Its universal object then defines the desired map $\Phi': W'\times V\dashrightarrow \bP^k$. 

More generally, with a fixed embedding $\overline M \hookrightarrow \bP^k$, we consider 
the map $W'\times V \dashrightarrow \bP^k \times V$, defined 
by $(w', v) \mapsto  \big(\Phi' (w', v) , v\big)$. 
After pulling back through the natural embedding $\overline M\times V  \hookrightarrow \bP^k \times V$, 
we find the scheme $W$ of finite type, 
equipped with the rational map $W\times V\dashrightarrow \overline M$
satisfying the conditions in \autoref{item:kelid2}.
\end{proof}
\begin{set-up}\label{prologue}
We work in the setting of \autoref{def:WB}, \autoref{not:smooth} and \autoref{def:ratWB}. 
Let $\overline H$ be a projective scheme equipped with a morphism $\mu: \overline H \to \overline M$
of finite type.
Set $\overline S$ to be a set of rational maps from $V$ to $\overline H$ whose composition 
with $\mu$ is generically finite (where defined) and weakly-bounded. 
\end{set-up}
\begin{lemma}[Pullback of parameterization schemes]
\label{lem:Factors}
In the setting of \autoref{prologue} there are a scheme $W_{\overline H}$ of finite type
and a rational map $\Psi: W_{\overline H} \times V \dashrightarrow \overline H$
such that for every $\phi\in \overline S$, there is a closed point 
$w_{\overline H}\in W_{\overline H}$ for which we have $\Psi_{ \{ w_{\overline H} \}  \times V } = \phi$.
\end{lemma}
\begin{proof}
Let $W$ and $\Phi: W\times V \dashrightarrow \overline M$ be as in \autoref{FACT2}.
Let us first assume that $\Phi$ and each $\phi\in S_V$ is a morphism. 
Now, consider the morphism $\Phi': W\times V \to \overline M\times V$ 
defined by $(w,b) \mapsto \big( \Phi(w,b) , b  \big)$. 
After pulling back through $\mu\times \id: \overline H\times V \to \overline M\times V$, 
we find the scheme $W_{\overline H}$ of finite type 
and the commutative diagram 
\begin{equation}\label{eq:PSI}
\xymatrix{
W_{\overline H} \times V \ar[rr]^{\sigma_{\overline H} \times \id} \ar[d]_{\Psi_{\overline H}}  \ar@/_2.5pc/[dd]_{\Psi}  
&&   W\times V \ar[d]^{\Phi'}  \ar@/^2.5pc/[dd]^{\Phi}  \\
\overline H\times {V}  \ar[d]_{\pr_1} \ar[rr]^{ \mu \times \id}  &&  \overline M \times  V  \ar[d]^{\pr_1} \\
\overline H            \ar[rr]^{\mu}  &&    \overline M   ,
}
\end{equation}
where $\Psi_{\overline H}$ and $\sigma_{\overline H}$ are the natural projection maps.

Now, let $\phi\in \overline S$. Since $(\mu\circ \phi)$ is weakly-bounded, by \autoref{FACT2} 
there is $w\in W$ for which we have $\Phi_{\{ w  \} \times V} = \mu\circ \phi$.
Let $Y\subseteq \overline H$ be the locally closed image of $V$ under $\phi$. 
Define $\mu_Y:= \mu \circ i_Y$, where $i_Y: Y \hookrightarrow \overline H$ is the 
natural inclusion map. These lead to the schemes $W_Y$ and $w_Y$ of finite type, 
together with an embedding $w_Y \hookrightarrow W_Y$, all fitting in the 
following commutative diagram of fiber products
\begin{equation}\label{eq:BIG}
\xymatrix{
w_Y \times V \ar@{^{(}->}[d] \ar[rrrr]^{\mu_Y' \times \id} &&&&  \{w\} \times V \ar@{^{(}->}[d] \\
W_Y\times V \ar[d]  \ar@{^{(}->}[rr]   && W_{\overline H}\times V \ar[d]^{\Psi_{\overline H}}  \ar[rr] &&   W\times V \ar[d]^{\Phi'}  \\
Y \times  V  \ar@/_2pc/[rrrr]_{\mu_Y\times \id}  \ar@{^{(}->}[rr]   &&  \overline H \times V  \ar[rr]^{\mu \times \id}  &&  \overline M \times V ,
}
\end{equation}
with $\mu'_Y$ denoting the natural projection map. 
By the assumption made on the elements of $\overline S$ (\autoref{prologue}), the map  
$\mu_Y$ is generically finite over its image. 
As such, $w_Y$ is zero-dimensional. 
On the other hand, by construction, the outer cartesian square in \autoref{eq:BIG}
diagonally factors through $\{w \}\times V \to Y\times V$, naturally defined by $\phi$.
By the universal property of fiber products, there is a morphism 
\begin{equation}\label{eq:STAR}
\sigma:  \{ w\} \times V \longrightarrow  w_Y \times V
\end{equation}
whose composition with $(\mu'_Y \times \id)$ is the identity, 
and when composed with $w_Y\times V \to Y \times V \overset{\pr_1}{\longrightarrow} Y$ coincides 
with $\phi$. In particular the image of $\sigma$ in \autoref{eq:STAR} is $\{w' \}\times V$, for a 
closed point $w'\in w_Y$. 
Since $w_Y\times V \hookrightarrow W_Y \times V \hookrightarrow W_{\overline H}\times V$ 
are embeddings, their composition identifies $\{w'\}\times V$ with 
$\{w_{\overline H} \} \times V$, for some $w_{\overline H}\in W_{\overline H}$.
Moreover, we have $\Phi_{ \{ w_{\overline H} \} \times V } =\phi$.

When elements of $\overline S$ and $\Phi$ are assumed to be only rational maps, 
for each $\phi\in \overline S$ we consider $\{ w \} \times V^\circ$, where $V^\circ$
denotes the complement of the indeterminacy locus of $\phi$, 
and repeat the above arguments using the following commutative diagram instead:
$$
\xymatrix{
w_Y\times V^\circ \ar[rrrr] \ar[d] &&&&  \{ w\} \times V^\circ \ar[d]  \\
Y \times V  \ar@{^{(}->}[rr]    &&  \overline H\times V  \ar[rr]  &&   \overline M \times V .
}
$$
\end{proof}
%
%
%
%
\section{Bounding maps to the parametrizing scheme of certain Hilbert functors}  
\label{Section3-WB} 
In this section we will be working in the context of the following set-up.

\begin{set-up}\label{setup:lem}
Let $\frM'_{h}$ be a bounded (\autoref{def:bounded}) and 
locally-closed (\autoref{def:openclosed}) 
moduli functor. 
We will work in the setting of  \ref{ss:Hilb2}, \ref{ss:good}
and \autoref{not:compact}. 
In particular we will assume that, for a suitable choice of integer $k$ and corresponding $d, l\in \bN$, 
$\frM'_{h}$ can be approximated by $\mathfrak H^{d,l,k}_{\frM'_h}$
as a distinguished moduli functor:
the universal object of $\mathfrak H^{d,l,k}_{\frM'_h}$ has a good minimal model 
equipped with a polarization over $\overline H\backslash D_{\overline H}$ such that 
the associated moduli functor $\mathfrak M''_{h''}$ satisfies \autoref{finiteAut} and \autoref{coarse}. 
Following the notations introduced in \autoref{not:compact}, the map $\mu: \overline H\dashrightarrow \overline M_{h''}''$ denotes
the induced moduli map, 
which, by construction, is a morphism over $\overline H \backslash D_{\overline H}$.
\end{set-up}

\begin{notation}\label{not:S}
In the situation of \autoref{setup:lem}, 
let $V$ be any smooth quasi-projective variety and $B$ a compactification as in \autoref{not:smooth}. 
\begin{enumerate}
\item\label{set1} Let $S$ be the set of all rational maps $\phi_V: V\dashrightarrow \overline H$ factoring 
through $\mathfrak H^{d,l,k}_{\frM'_h}$ over an open subset $V'\subseteq V$ (\autoref{term}), 
and arising from some $(f_U: U\to V) \in \frM'_h(V)$.
\end{enumerate}
\end{notation}

\begin{lemma}[Weak bounds arising from minimal models of universal families]
\label{lem1}
In the setting of \autoref{not:S} 
assume that $\phi_V(V')\not\subseteq D_{\overline H}$. 
Then, for any ample line bundle $\lambda_{\overline M''_{h''}}$ on $\overline M''_{h''}$ 
there is a weak bound $b_{\lambda_{\overline M''_{h''}}}$ with respect to which 
$(\mu\circ \phi_V): B \dashrightarrow \overline M''_{h''}$ is weakly-bounded, for every $\phi_V\in S$.  
\end{lemma}

Before proving \autoref{lem1} we state its key consequences. 
We will keep using the notations and definitions in \ref{ss:Hilb2} and \ref{ss:good}.

\begin{proposition}
\label{prop:key}
Let $S$ be as in \autoref{not:S}. 
There are a scheme $W_{\overline H}$ of finite type and a rational map 
$\Psi: W_{\overline H}\times B \dashrightarrow \overline H$ 
such that for every $\phi_V\in S$, for which $\Var(f_U)= \dim V$, there is a closed point $w_{\overline H}\in W_{\overline H}$ 
satisfying $\Psi_{\{ w_{\overline H}\} \times V} = \phi_{V}$.
\end{proposition}
\begin{proof}
By induction on $\dim (H)$ we may assume that $\phi_{V}(V') \not\subseteq D_{\overline H}$.
That is, when $\phi_{V}(V')\subseteq D_{\overline H}$ 
we consider the pullback of $\X$ to the locus of $D_{\overline H}$ in $\overline H$ (see \ref{ss:good} for notations)
and repeat the constructions of \ref{ss:good} with $\X$ being replaced by this latter pullback family.

Now, according to \autoref{lem1} we know that for every $\phi_V\in S$, the composition 
$(\mu\circ \phi_V): B\dashrightarrow \overline M''_{h''}$ is weakly-bounded. 
Existence of $W_{\overline H}$ and $\Psi$ is thus guaranteed by \autoref{FACT2}, 
combined with \autoref{lem:Factors}, using the maximality assumption on the variation (see \autoref{prologue})
and the compatibility property of $M''_{h''}$ \autoref{coarse}.
\end{proof}

\subsection{Application to birational boundedness}
\begin{definition}[Strongly birationally bounded]
\label{def:strong}
For a birationally bounded $\mathfrak C_V$ in the setting 
of \autoref{def:BB}, if we further assume that there is a line bundle $\sL_{\mathcal Y}$ on $\mathcal Y$ 
such that for every polarized $(f^0:U\to V, \sL) \in \mathfrak C_V$ there is an open subset $V^0\subseteq V$ 
 verifying $\big( \mathcal Y, \sL_{\mathcal Y}\big)_{ \{w\} \times V^0 }   \simeq_{V^0} (U, \sL)_{V^0}$, 
as polarized schemes, 
we then say $\mathfrak C_V$ is strongly birationally bounded. 
\end{definition}

\begin{corollary}
\label{cor:key}
Assuming that $\mathfrak M'_h$ satisfies the assumptions made 
in \autoref{setup:lem}, 
the subset $\mathfrak C_V  \subset \mathfrak C_{V, \mathfrak M'_h}$ 
(\autoref{not:frakC}) consisting of families with maximal variation is (strongly) birationally bounded. 
\end{corollary}

\begin{proof}
By the constructions in \ref{ss:Hilb2} this follows directly from \autoref{prop:key}:
take $\mathcal Y$ (in \autoref{def:BB} and \autoref{def:strong})
to be the pullback of $\X$ (\ref{ss:good}) by the map $\Psi$ in \autoref{prop:key}, over the complement of its indeterminacy. 
The line bundle $\sL_{\mathcal Y}$ is similarly taken to be the pullback of the 
universal one $\sL_{\X}$ for $\mathfrak H^{d,l,k}_{\frM'_h}$, which by construction satisfies  
the isomorphism \autoref{isom:pb2}.
\end{proof}

\noindent
\textbf{Proof of \autoref{thm:main}.}
This is a direct consequence of \autoref{rk:lc} and \autoref{claim:good} (see \autoref{ex:good})
together with \autoref{cor:key}.

\subsection{Proof of \autoref{lem1}.}
Let $C^\circ \subseteq V$ be any smooth quasi-projective curve together with a smooth compactification $C \to B$.
Assume that $C':= C^\circ \cap V'\neq \emptyset$, 
i.e., we have a morphism $\phi_{C'}: C'\to H$ that factors through $\mathfrak H^{d,l,k}_{\frM'_h}$, 
which arises from the pullback of $f_U$ to $C^\circ$.
We denote the latter by $(f^\circ: U\to C^\circ)\in \frM'_{h}(C^\circ)$.
Let $\phi_{C}: C\to \overline H$ be the unique extension of $\phi_{C'}$.
Assume that $\phi_C(C) \not\subseteq D_{\overline H}$ 
(i.e. $\phi_C(C)$ is not contained in the indeterminacy locus of $\mu: \overline H\dashrightarrow \overline M''_{h''}$).

After replacing $\overline H$ by a smooth projective birational model $\widehat H$, we remove 
the indeterminacy of $\mu: \overline H \dashrightarrow \overline M''_{h''}$ and find the 
morphism $\widehat \mu: \widehat H\to \overline M''_{h''}$. 
As $\phi_C(C) \not\subseteq D_{\overline H}$, the induced rational map 
$C \dashrightarrow \widehat H$ uniquely extends to a morphism 
$\widehat \phi_C : C\to \widehat H$. 
With no loss of generality we may replace $\overline H$ by $\widehat H$ and 
use $\phi_C: C\to \overline H$ 
and $\overline \mu: \overline H \to \overline M''_{h''}$ to respectively denote $\widehat \phi_C$ 
and $\widehat \mu$. Let us denote the image of $\overline \mu$ by $\overline M$
and define $\overline \mu_C := \overline \mu\circ \phi_C$. 

Next, let $\phi_T: T\to \overline H$ be a desingularization 
of a general subvariety in $\overline H$ that is not contained in $\supp(D_{\overline H})$, 
and such that the 
naturally induced morphism $\overline \mu_T: = \phi_T\circ\overline \mu: T \to \overline M$ 
is generically finite and surjective. 
Set $T^\circ:= T \backslash \phi_T^{-1}(D_{\overline H})$. 
Using this construction, and following \autoref{def:ratWB}, to show \autoref{lem1} 
it suffices to prove the following reformulation of the lemma.

\begin{reformulation}\label{reform}
For any ample line bundle $\lambda_{\overline M''_{h''}}$ on $\overline M''_{h''}$, 
there is $b_{\lambda_{\overline M''_{h''}}}$ such that, 
for any $C^\circ \subseteq V$, 
any morphism $\overline \mu_C: C\to \overline M$ as above is weakly-bounded with 
respect to $b_{\lambda_{\overline M''_{h''}}}$, if 
$\overline \mu_C(C) \not\subseteq \overline M\backslash \overline \mu_T(T^\circ)$.
\end{reformulation}
The rest of the proof is devoted to establishing \autoref{reform}.

\smallskip

\noindent
\emph{Further geometric constructions:}
Let $f_{\X'}: \X' \to \overline H\backslash D_{\overline H}$ 
be as in \autoref{not:compact} and 
 $\sL'$ the line bundle on $\X'$ which we have used to find $\frM''_{h''}$
satisfying 
\autoref{finiteAut} and \autoref{coarse}. 
Denote the pullback of $f_{\X'}$ to $T^\circ$ (as polarized schemes) 
by $f'_{T^\circ}: \X'_{T^\circ} \to T^\circ$. 
By \autoref{coarse} we have $\Var(f'_{T^\circ}) = \dim T$. 
Now, let $f_T:\X_T\to T$ denote the pullback of $\overline \X \to \overline H$
via $\phi_T$ (see \autoref{not:compact}). 
After replacing $f_T$ by a locally-stable reduction (of its smooth locus) \autoref{KolRed}, 
existence of which is guaranteed by the assumption on the universal embedded object
(\autoref{ss:good}),
we may assume that $f_T$ is locally-stable (\autoref{def:LS}).

Now, let $C_T$ be the normalization of the main component of $C\times_{\overline M} T$
and denote the resulting naturally induced morphisms by $\gamma: C_T\to C$ 
and $\psi: C_T\to T$.
We note that by construction $\gamma$ is surjective and, as $C$ is smooth, flat.
Let $f: X\to C$ be a smooth compactification of
$f^0:U\to C^0$ and set $X_{C_T}$ to be a desingularization of 
$X \times_C C_T$. Let $Y$ be a desingularization
of the main component $\widehat\X_T$ of $\X_T\times_T C_T$, 
with the naturally induced morphisms
$\widehat f_T: \widehat\X_T\to C_T$, $\psi':\widehat \X_T\to \X_T$ 
and $f_Y:Y\to C_T$. 
We summarize these constructions in the following 
commutative diagram: 
\begin{equation}\label{eq:Big}
\xymatrix{
U \ar[r]^{\subseteq} \ar[d]_{f^0} & X  \ar[d]^f   &   X_{C_T}  \ar[l] \ar[dr]^{f_{C_T}}  &&  Y  \ar[r] \ar[dl]_{f_Y}  &  \widehat\X_T \ar[dll]^{\widehat f_T}  
\ar[r]^{\psi'}  & \X_T \ar[d]_{\tiny{\txt{loc.\\ stable}}}^{f_T}   \\
C^0 \ar[r]^{\subseteq} & C    &&  C_T \ar[ll]_{\gamma}    \ar[rrr]^{\psi}  &&&  T .
}
\end{equation}
Furthermore, with $ \overline \mu_{C_T}: =  \overline \mu_T \circ\psi : C_T \to \overline M$, we have a commutative 
diagram: 
\begin{equation}\label{eq:commaps}
\xymatrix{
C  \ar[drr]_{\overline \mu_C} &&   C_T  \ar[ll]_{\gamma}  \ar[rr]^{\psi} \ar[d]^{\overline \mu_{C_T}} &&   T \ar[dll]^{\overline \mu_T} \\
     &&    \overline M  .
}
\end{equation}
\begin{claim}\label{claim:bir}
After replacing $C_T$ by a finite covering $\widehat C_T\to C_T$, 
$Y$ is birational 
to $X_{C_T}$ over $C_T$. 
\end{claim}
\noindent
\emph{Proof of Claim~\ref{claim:bir}.}
Set $Y'$ to be the pullback of $\X'_{T^0}$ over the open subset of $C_T$ mapping to 
$T^0$. Noting that $U_{C'}:= \X \times_H C'$, over the complement of $D_{\overline H}$ set 
$$
U'_{C'}  := \X' \times_H C'    \;\;\;\;\;\;\;\;\; \text{and}   \;\;\;\;\;\;\;\;\;    U'_{C_T}: = U'_{C'} \times_{C'} C'_T  ,
$$
where $C'_T:= \gamma^{-1}(C')$.
Now, after endowing these new families with the pullback of $\sL'$, we find that by construction \autoref{eq:commaps}
the induced moduli maps to $\overline M$ for 
$Y'\to C_T$ and $U'_{C_T}\to C'_T$ coincide. 
This implies that the general fibers of these two families are isomorphic as polarized schemes. 
Since the polarized varieties in $\frM''_{h''}$ have finite automorphism groups \autoref{finiteAut}, 
using \autoref{fact:LiftIsom}, it follows that
there are a finite covering $\widehat C_T\to C_T$ and  
an isomorphism between $(U'_{C_T})_{\widehat C_T}$ 
and $Y'_{\widehat C_T}$ over an open subset of $\widehat C_T$. 
This implies that $Y_{\widehat C_T}$ and $(X_{C_T})_{\widehat C_T}$ 
are birational. \qed

\smallskip

\noindent
\emph{Bounding maps to $\overline M_{h''}''$:}
Thanks to Karu~\cite[Thm.2.5]{Kar00} (see also Kawakita~\cite{Kawakita} and Kov\'acs--Schwede~\cite{KS13}) we know that 
$\X_T$ and $\widehat \X_T$ have only canonical singularities.
Moreover, since $f_T$ is locally stable, by definition there is $m'\in \bN$ such that $\omega^{[m']}_{\X_T/T}$ is invertible. 

According to \autoref{KovTaj} to prove \autoref{reform} it suffices to prove the following claim.
\begin{claim}\label{claim:main}
Let $\lambda_{\overline M}: = \lambda_{\overline M''_{h''}}|_{\overline M}$. 
For every sufficiently large integer multiple $m$ of $m'$, there is $k_m\in \bN$ such that 
the inequality
$$
\deg \big(  (\overline \mu_C)^* \lambda_{\overline M} \big)  \leq k_m\cdot \deg (f_*\omega^m_{X/C})
$$
holds for any smooth projective curve $f: X\to C$ and $\overline \mu_C$ as above. 
As a consequence, for every such $m$, with the polynomial $b_m$ as in \autoref{KovTaj}, $k_m\cdot b_m$ 
is the required weak bound $b_{\lambda_{\overline M''_{h''}}}$ in \autoref{reform}.
\end{claim}
\noindent
\emph{Proof of Claim~\ref{claim:main}.}
As $f_T$ is locally-stable, according to \autoref{rem:BC}
we have:
\begin{enumerate}
\item \label{BS} $\psi^* \big( (f_T)_* \omega^{[m]}_{\X_T/T}\big) \simeq  
                         (\widehat f_T)_*\omega^{[m]}_{\widehat \X_T/C_T}$, for any $m\in \bN$.
\item \label{PB} $(\psi')^* \omega^{[m]}_{\X_T/T} \simeq \omega^{[m]}_{\widehat \X_T/C_T}$. 
In particular $\omega^{[m']}_{\widehat \X_T/C_T}$ is invertible.
\end{enumerate}
Therefore, with $\widehat \X_T$ having only canonical singularities, we have 
\begin{equation}\label{eq:same}
(f_Y)_* \omega^m_{Y/C_T} \simeq (\widehat f_T)_* \omega^{[m]}_{\widehat \X_T/ C_T} , 
\end{equation}
for every $m\in \bN$ as in \autoref{claim:main}. Moreover, we have:
\begin{subclaim}\label{SC}
$\det \big((f_T)_* \omega^{[m]}_{\X_T/T}\big)$ is a big line bundle. 
\end{subclaim}
\noindent
\emph{Proof of Subclaim~\ref{SC}.}
Since $\X_T$ has only canonical singularities, for any $m$ as above we have 
$$
(\wtilde f_T)_* \omega^m_{\wtilde \X_T/T} \simeq  (f_T)_* \omega^{[m]}_{\X_T/T} ,
$$
where $\wtilde \X_T\to \X_T$ is a resolution and $\wtilde f_T$ is the induced family. 
Therefore, we have 
$$
\det \big((\wtilde f_T)_* \omega^m_{\wtilde \X_T/T}\big) \simeq  \det \big((f_T)_* \omega^{[m]}_{\X_T/T} \big). 
$$
The rest now follows from $\Var(f_T)=\Var(\wtilde f_T)=\dim T$
and \cite[Thm.~1.1.(i)]{Kawamata85}. \qed

Now, let $k_m\in \bN$ be sufficiently large so that 
$h^0\big(  (\det (f_T)_* \omega^{[m]}_{\X_T/T})^{k_m} \otimes (\overline \mu_T)^*\lambda_{\overline M}^{-1}  \big)\neq 0$, i.e., 
there is an injection 
\begin{equation}\label{inject}
(\overline \mu_T)^*\lambda_{\overline M}  \hooklongrightarrow  \Big(\det (f_T)_* \omega^{[m]}_{\X_T/T}\Big)^{k_m} .
\end{equation}
Next, we note that, with $f_T$ being locally-stable, $\bigotimes^{r_m}(f_T)_* \omega^{[m]}_{\X_T/T}$ 
is reflexive \cite[Lem.~2.8]{VZ02}, where $r_m:= \rank( (f_{T})_* \omega^{[m]}_{\X_T/T} )$.
Therefore, it naturally contains $\det (f_T)_* \omega^{[m]}_{\X_T/T}$ as a direct factor. 
Same holds for $\bigotimes^{r_m}(\widehat f_T)_* \omega^{[m]}_{\widehat \X_T/C_T}$.
As such, the isomorphism \autoref{BS} induces the isomorphism 
$$
\psi^* \det (f_T)_* \omega^{[m]}_{\X_T/T} \simeq \det (\widehat f_T)_* \omega^{[m]}_{\widehat \X_T/C_T} .
$$
By \autoref{eq:same} we then find 
$$
 \det\big((f_Y)_*\omega^m_{Y/C_T}\big)  \simeq  \psi^* \det\big( (f_T)_* \omega^{[m]}_{\X_T/T}\big)  .
$$
On the other hand, as $Y$ and $X_{C_T}$ are birational over $C_T$ by \autoref{claim:bir}, 
we have 
$$
\det \Big(  (f_Y)_* \omega^m_{Y/C_T} \Big)  \simeq  \det \Big(  (f_{C_T})_* \omega^m_{X_C/C_T} \Big) .
$$
Furthermore, with $\gamma$ being flat, we find that
$$
 \det \Big(  (f_{C_T})_* \omega^m_{X_C/C_T} \Big)  \subseteq \gamma^* (\det f_*\omega^m_{X/C} ),
$$
 cf.~\cite[\S 3, p.~336]{Viehweg83}.
 Using \autoref{inject} this implies that 
$$
\underbrace{\psi^* \overline \mu_T^* (\lambda_{\overline M})}_{\;\; \overset{\autoref{eq:commaps}}{\simeq}  \;\; \gamma^*\overline \mu_C^* \lambda_{\overline M}}
 \hooklongrightarrow  \psi^*\big( \det(f_T)_* \omega^{[m]}_{\X_T/T} \big)^{k_m}  
\subseteq  \Big( \gamma^* (\det f_*\omega^m_{X/C} \Big)^{k_m}  ,
$$
which establishes the claim.  \qed

\section{Boundedness for families: Proof of \autoref{main:maxvar}}
\label{Section5-Bounded}

\subsubsection{Compactifiability of $\mathfrak M^{\sa}_{\overline h}$}

\begin{definition}[\protect{Quotient stacks, cf.~\cite[Def.~4.1]{Kresch}}]
A separated Deligne--Mumford stack of finite type $\mathcal M$ is a quotient stack, 
if $\mathcal M \simeq \big[H/G\big]$, for an algebraic space $H$ and linear algebraic group $G$.
\end{definition}

\begin{proposition}\label{prop:compact}
$\mathcal M^{\sa}_{\overline h}$ is a compactifiable Deligne--Mumford stack. 
\end{proposition}

\begin{proof}
As $M_{\overline h}^{\sa}$ is quasi-projective, according to \cite[Thm.~5.3]{Kresch} it is enough to show that $\mathcal M_{\h}^{\sa}$ is  
a quotient stack. The latter is however a direct consequence of Viehweg's constructions in 
\cite[\S 7.5]{Viehweg95}. 
More precisely, one uses the (doubly polarized) Hilbert functor $\mathfrak H_{\mathfrak M^{\sa}_{\h}}$
of embedded schemes in $\mathfrak M^{\sa}_{\h}$, which is
representable by a quasi-projective $H$ \cite[Thm.~1.52]{Viehweg95}. 
Using the locally-closedness of 
of $\mathfrak M^{\sa}_{\h}$ (\autoref{rk:lc}), one finds a locally-closed 
subscheme $H^\circ \subseteq H$ over which 
the universal family of $\mathfrak H_{\mathfrak M^{\sa}_{\h}}$ parametrizes 
 all objects of $\mathfrak M^{\sa}_{\h}$. 
After replacing $H^\circ$ by an open subset if necessary \cite[p.~52]{Viehweg95}, we then find that 
$H^\circ$ is naturally equipped with the action of a linear group $G$
(of the form $\mathrm{PGL}_l\times \mathrm{PGL}_m$, for some $l,m\in \bN$ defined by $\mathfrak H_{\frM_{\h}^{\sa}}$) 
\cite[p.~222]{Viehweg95}. With the action of $G$ being proper, 
having finite stabilizers \cite[Lem.~7.6]{Viehweg95}, 
we find that $\mathcal M^{\sa}_{\h}$ 
is a separated Deligne--Mumford stack isomorphic to
$\big[H^\circ/G\big]$.
\end{proof}

\subsubsection{Coarse boundedness}
According to \autoref{thm:KL}, by using \autoref{prop:compact}, in order to prove \autoref{main:maxvar} 
it is enough to show the following proposition. 

\begin{proposition}\label{prop:CB}
The subset of $\mathfrak C_{V, \mathfrak M^{\sa}_{\h}}$ consisting of all maximally-varying 
families is coarsely-bounded in the sense of \autoref{def:CB}. 
\end{proposition}

\begin{proof} We will work in the setting of \autoref{ss:Hilb2}
and \autoref{ss:good}. We divide the proof into three steps.

\noindent
\emph{Step.~1: Bounding maps to an auxiliary parametrization space:}
 As in \autoref{rk:compare} we define $h(x):= \overline h(x,x)$
 and consider $\frM_h^{\sa}$ (\autoref{def:gmmMod}), which is 
 locally-closed and bounded (\autoref{rk:lc}).
 Let $d,l,k\in \bN$ be as in \autoref{ss:Hilb2} and set
$f_{\X}: (\X, \sL_{\X}) \to H$ be the universal object 
of $\mathfrak H_{\frM_h^{\sa}}^{d, l,k}$. 
We may assume with no loss of generality that $H$ is reduced and irreducible. 
By the construction of $\mathfrak H_{\frM_h^{\sa}}^{d,l, k}$ we know that, for every $\mu_V: V\to M_{\overline h}^{\sa}$ 
arising from $(f_U: U\to V, \sL)\in \frM^{\sa}_{\h}(V)$,
 there is an open subset $V'\subseteq V$ equipped with a morphism 
 $\phi_{V'}: V'\to H$ that factors through $\mathfrak H^{d,l,k}_{\frM_h^{\sa}}$ (\autoref{term}), 
defined by the restriction of $(f_U: U\to V, \omega_{U/V}\otimes \sL)$ to $V'$ (\autoref{rk:compare}).
That is, by pulling back via $\phi_{V'}$, we have 
\begin{equation}\label{PullIsom}
\big(  f_{\X} : \X\to H , \sL_{\X}  \big)_{V'}  \simeq_{V'}  \Big( f_U: U\to V, \omega_{U/V} \otimes \sL  \Big)_{V'}.
\end{equation}
Let $\phi_V: V\dashrightarrow \overline H$ denote the corresponding 
 rational map to a compactification $\overline H$ of $H$. 
 By \autoref{claim:good} we know that $\mathfrak H^{d,l,k}_{\frM_h^{\sa}}$ is
 a distinguished moduli functor. Therefore, 
 the conclusions of \autoref{prop:key} are valid 
 for $\overline H$ and any such rational map $\phi_V$, 
 as long as $\Var(f_U)= \dim V$. 
 
 \noindent
 \emph{Step.~2: Simultaneous factorization of coarse maps:}
 Let $H_{\reg}$ denote the smooth locus  of $H$ and define 
 $\X_{\reg}$ to be the pullback of $\X$ to $H_{\reg}$
 (in particular $\X_{\reg}$ is smooth). 
 By induction on $\dim (H)$ we may assume that $\phi_V(V')\not\subseteq (H\backslash H_{\reg})$.
 As such, the locus of $H_{\reg}$ over which 
 $\sM_{\X_{\reg}}: = ( \omega^{-1}_{\X_{\reg} /H_{\reg}}  \otimes \sL_{\X_{\reg}} )$ is relatively-ample is non-empty. 
 Let $H^\circ \subseteq H_{\reg}$ be the maximal open subset over which 
 $\sM_{\X_{\reg}}$ is relatively-ample. 
 By construction, for every $\phi_V$ and $V'\subseteq V$ as above, there is 
 an open subset $V''\subseteq V'$, with $\phi_V(V'')\subseteq H^\circ$, such that
 \begin{equation}\label{heart}
 ( U, \sL )_{V''}  \simeq_{V''}  (\X^\circ, \sM_{\X^\circ})_{V''}   ,
 \end{equation}
 where $\X^\circ$ denotes the restriction of $\X$ to $f_{\X}^{-1}(H^\circ)$.
For every $\alpha, \beta\in \bN$, 
 $\omega_{\X^\circ/H^\circ}^{\alpha}\otimes \sM_{\X^\circ}^{\beta}$ is flat over $H^\circ$.
 In particular, by applying \cite[Thm.~III.9.9]{Ha77} (see also \cite[Prop.~2.1.2]{MR2665168}) 
 to $\omega_{\X^\circ/H^\circ}^{\alpha}\otimes \sM_{\X^\circ}^{\beta}$, 
 for every fixed $\alpha$, 
we find that there is $\wtilde h\in \bQ[x_1, x_2]$ for which the function 
$$
t \mapsto  \chi\big(  \omega_{\X_t}^{\alpha}  \otimes \sM_{\X_t}^{\beta}   \big)  = \wtilde h(\alpha, \beta)  \; , \;  \forall \alpha, \beta\in \bN
$$ 
remains constant for all closed points $t\in H^\circ$. On the other hand, with $(f_U: U\to V, \sL)\in \frM_{\h}^{\sa}(V)$ 
as above, we have 
$$
\Big(  U, \omega^{\alpha}_{U/V}  \otimes \sL^{\beta}  \Big)_{V''}   \simeq_{V''}   
    \Big( \X^\circ,  \omega^{\alpha}_{\X^\circ/H^\circ}\otimes \sM^{\beta}_{\X^\circ}  \Big)_{V''}
$$
(using \autoref{heart}). This implies that $\h = \wtilde h$. 

Now, let $\mu_{\h}^{\sa} : H^\circ\to M_{\h}^{\sa}$ be the moduli map 
induced by $\big( \X^\circ \to H^\circ, \sM_{\X^\circ}  \big)$, 
with $\overline \mu_{\h}^{\sa}: \overline H \longrightarrow \overline M_{\h}^{\sa}$ 
denoting the corresponding rational map to a compactification $\overline M_{\h}^{\sa}$ 
of $M_{\h}^{\sa}$. For every $\phi_V$ as above we thus have a rational 
map $\overline\mu_{\h}^{\sa}\circ \phi_V : V \dashrightarrow  \overline M_{\h}^{\sa}$ 
fitting in the diagram: 
\begin{equation}\label{eq:Diag}
\xymatrix{
\big(U, \sL)  \ar[d]_{f_U}^{\in \frM_{\h}^{\sa}(V)}  &&&  \big( \X^\circ, \sM_{\X^\circ} \big) \ar[d]  \ar[rrr]  &&&  
  \tiny{{\begin{array}{@{}c@{}} \text{univ. obj. of}  {}\\ \mathfrak H^{d,l,k}_{\frM_h^{\sa}}  \\ \text{twisted by \;} \omega^{-1}_{\X^\circ/H^\circ}  \end{array}}}  \\
V \ar[rrrd]_{\mu_V}    \ar@{-->}@/^2pc/[rrrr]^{\phi_V}   &  \ar@{_{(}->}[l]_(.4){\supseteq}    V'' \ar[rr]^{\phi_{V'}} &&   
H^\circ \ar[d]^{\mu_{\h}^{\sa}} \ar[r]^(.4){\subseteq}  &  \overline H \ar@{-->}[d]^{\overline \mu_{\h}^{\sa}}  \ar@{-->}[r]   
                            &  \overline M_{h''}''  \ar[r] & \text{\tiny{as in \autoref{ss:good}}}  \\
&&&                                                                    M^{\sa}_{\h}  \ar@{^{(}->}[r]       &    \overline M^{\sa}_{\h}
}
\end{equation} 
 
 By \autoref{heart} $\overline \mu_{\h}^{\sa}\circ \phi_V$ and $\mu_V$ coincide over $V''$. 
 It thus follows that all maps in \autoref{eq:Diag} commute. In particular the composition 
 $\overline \mu_{\h}^{\sa}\circ \phi_V = \mu_V$ 
 is a morphism from $V$ to $M_{\h}^{\sa}$. 
 
 \noindent
 \emph{Step.~3: Parameterizing factorized coarse maps:}
 Noting that every $\phi_V$ as above is a member of the set $S$ 
 in \autoref{not:S}, with $\frM'_h:= \frM_h^{\sa}$, by \autoref{prop:key}
 there are a scheme $W_{\overline H}$ of finite type 
 and a rational map $\Phi_{\overline H}: W_{\overline H}\times V \dashrightarrow \overline H$
 such that for every such $\phi_V$ we have 
 \begin{equation}\label{eq:point}
 \phi_V = (\Phi_{\overline H})_{ \{ w_{\overline H}  \}  \times V } \;\; , 
\text{\;\;\;\; for some $w_{\overline H} \in W_{\overline H}$ .}
 \end{equation}
 Therefore, after composing with $\overline \mu^{\sa}_{\h}$
$$
\Psi_{\overline H}: = \overline \mu_{\h}^{\sa} \circ \Phi_{\overline H} : W_{\overline H} \times V \dashrightarrow \overline M^{\sa}_{\h} ,
$$
we find that for every $\mu_V: V\to M^{\sa}_{\overline h}$ 
factoring through $\mathcal M^{\sa}_{\overline h}$, 
via a maximally varying family, we have 
$(\Psi_{\overline H} )_{ \{ w_{\overline H} \}  \times V } =  \overline \mu_h^{\sa} \circ \phi_V = \mu_V$.

\begin{claim}\label{claim:final}
With no loss of generality we may assume that $\Psi_{\overline H}$ is a morphism.
\end{claim}

\noindent
\emph{Proof of \autoref{claim:final}.} Let $I\subset W_{\overline H}\times V$ denote the 
reduced closed subscheme underlying the indeterminacy locus of $\Psi_{\overline H}$ 
and consider its image $I'$ under the natural projection 
$\pr_1: W_{\overline H} \times V\to W_{\overline H}$. Noting that  
for every $w_{\overline H}\in W_{\overline H}$ as in \autoref{eq:point} $\Psi_{ \{w_{\overline H} \} \times V  }$ is a morphism
to $M_{\h}^{\sa}$, 
we find that $I' \neq W_{\overline H}$.
With $I'\subseteq W_{\overline H}$ being constructible \cite[Ex.II. 3.19]{Ha77} and $W_{\overline H}$ of finite type, 
by replacing $W_{\overline H}$ with $W_{\overline H}\backslash I'$ we may assume without loss of generality 
that $\Psi_{\overline H}$ is a morphism. 
\qed

We now consider the morphism 
$\Psi'_{\overline H}: W_{\overline H} \times V \to \overline M_{\h}^{\sa}\times V$ 
defined by $(w_{\overline H}, v) \mapsto \big(\Psi_{\overline H}(w_{\overline H} , v), v \big)$. 
After pulling back $\Psi'_{\overline H}$ via the natural open immersion 
$M_{\overline h}^{\sa} \times V \hookrightarrow \overline M_{\h}^{\sa} \times V$, 
we find a scheme $W^{\sa}_{\overline H}$ and a morphism $W^{\sa}_{\overline H}\times V\to M_{\h}^{\sa}$ 
of finite type, which respectively play the role of $W$ and $\Phi$ in \autoref{def:CB}.\end{proof}

\noindent
\emph{Proof of \autoref{main:maxvar}.}
This is a direct consequence of \autoref{thm:KL}, \autoref{prop:compact} and \autoref{prop:CB}.

\section{Appendix: Polarizations compatible with variation}
\label{Section5-Appendix}
In this section we provide a proof for \autoref{lem:comp} using~\cite[\S3.2]{Taji20}.

\begin{set-up}\label{KawSetup}
Let $f: U' \to V$ be a relative good minimal model.
According to \cite[Lem.~7.1]{Kawamata85} there are smooth quasi-projective varieties $\overline V$ and $V''$, 
a surjective morphism $\rho: \overline V \to V''$ and a surjective, generically finite morphism 
$\sigma: \overline V\to V$ with a projective morphisms $f'': U''\to V''$:
\begin{equation}\label{diag:1}
\xymatrix{
U' \ar[d]^{f} &&&         &&  U'' \ar[d]^{f''}    \\
V        &&&         \overline{V}  \ar[lll]_{\sigma}^{\text{generically finite}}  \ar[rr]^{\rho} &&   V''   , 
}
\end{equation}
satisfying the following properties. 
\begin{enumerate}
\item \label{Kaw1} Over an open subset $\overline V^\circ\subseteq \overline V$ the morphism $\sigma$ is finite and \'etale, and 
\item \label{Kaw2} we have 
$$
 \overline U^\circ : =U'' \times_{V''} \overline{V}^\circ  \simeq_{\overline{V}^\circ}   U' \times_V  \overline{V}^\circ,
$$
with $\rho':\overline U^\circ \to U''$ and $\sigma': \overline U^\circ \to U'$ being the natural projections.
 
\item \label{Kaw3} For every closed point $t \in \overline V^\circ$ the kernel of $(d_t\rho \circ  d_t\sigma^{-1})$ coincides with the kernel of 
the Kodaira--Spencer map for $f: U'\to V$ at $v= \sigma(t)$, where 
$d_t \rho$ and $d_t \sigma$ are the differentials of $\rho$ and $\sigma$ at $t$.
\end{enumerate}

This provides the following alternative description of the notion of variation for families of good minimal models.

\begin{def-thm}[\protect{\cite[Lem.~7.1, Thm.~7.2]{Kawamata85}}]
\label{def:KawVAR}
For every family of good minimal models $f: U'\to V$, the algebraic
closure $K:= \overline{\bC(V'')}$ is the (unique) minimal closed field of definition for $f$, that is 
$\Var(f) = \dim V''$.
\end{def-thm}
\end{set-up}

One can observe that \cite[Lem.~7.1, Thm.~7.2]{Kawamata85} in particular implies that, 
for families of good minimal models, variation in the sense of \ref{def:var} can be measured, at least generically (over the base), 
by the Kodaira--Spencer map. 
Of course this property fails in the absence of the good minimal model assumption
(for example one can construct a smooth projective family of non-minimal varieties of 
general type with zero variation and generically injective Kodaira--Spencer map).
For future reference, we emphasize and slightly extend this point in the following observation. 

\begin{observation}\label{OBSERVE}
We will work in the situation of Set-up~\ref{KawSetup}.
\begin{enumerate}
\item\label{item:trivial} For every smooth subvariety $T\subseteq \overline{V}^0$, with $\rho(T)$ being a closed point, 
the family $\overline{U}_T^0 \to T$ is trivial. In particular, if $\Var(f) =0$, then $f$ is generically (over V) isotrivial.
\item\label{item:trivialPol} For every $T\subseteq \overline{V}^0$ as in Item~\ref{item:trivial} 
and line bundle $\sL''$ on $U''$, the polarized family $\big(  \overline{U}^0_T\to T, (\rho')^* \sL''  \big)$ is trivial. 

\item\label{item:KSVar} $\Var(f)= \dim V$, if and only if the Kodaira--Spencer map for $f$ is generically injective
(in $V$). 
\end{enumerate}
 To see this, we may assume that $\overline V^0 = \overline V$. Set $v'':= \rho(T) \in V''$. 
By the assumption we have 
$$
\overline U_T \cong T \times_{\bC} F , 
$$
where $F:= U''_{v''}$ (which shows Item~\ref{item:trivial}). Thus, over $T$, $\rho'$ coincides with 
the natural projection $\pr_2: T\times_{\bC} F\to F$. Clearly, $(T \times_{\bC} F, \pr_2^*\sL''_{v''})$ is trivial, 
showing \autoref{item:trivialPol}. Item~\ref{item:KSVar} follows from \autoref{item:trivial}, \autoref{Kaw3} 
and \autoref{def:KawVAR}.
\end{observation}

\noindent
\emph{Proof of \autoref{lem:comp}:} 

We start by considering Diagram~\autoref{diag:1}. 
In \autoref{KawSetup} we may assume that $\overline V= \overline V^\circ$. 
Denote $U'\times_V \overline V$ by $\overline U$ and set $\overline f: \overline U \to \overline V$ to be the pullback family.
Using Item~\ref{Kaw2}, generically, the morphism $f''$ is a family of good minimal models, that is  
after replacing $V$ by an open subset $V_{\eta}$ we can assume that $f''$ is a relative good minimal model and flat. 
Let $\sL''$ 
be a choice of line bundle, as in \autoref{rk:LineExists}, so that 
$(f'': U''\to V'', \sL'') \in \frM_{h''}^{[N]}(V'')$, for some $h''\in \bQ[x]$. 
By $\mu_{V''}: V'' \to M_{h''}^{[N]}$ we denote the induced moduli map. 
We may assume that $\sigma$ is Galois, noting that if $\sigma$ is not Galois, we can replace 
it by its Galois closure and replace $\rho$ by the naturally induced map.
Define $G:= \Gal(\overline V/V_{\eta})$.

Now, we define $\sL''_{\overline V} := (\rho')^*\sL''$ and consider the $G$-sheaf 
$\bigotimes_{g\in G} g^*\sL''_{\overline V} \cong(\sL''_{\overline V})^{|G|}$ 
(see for example \cite[Def.~4.2.5]{MR2665168} for the definition). 
As $\sigma'$ is \'etale, the stabilizer of any point $\overline u\in \overline U$ 
is trivial (and thus so is its action on the fibers of $\sigma'$). Consequently, the above $G$-sheaf 
descends~\cite[Thm.~4.2.15]{MR2665168}. 
That is, there is a line bundle $\sL$ on $U_{V_{\eta}}$ 
such that 
$$
(\sigma')^* \sL \cong \bigotimes_{g\in G} g^* \sL''_{\overline V} .
$$
Therefore, we have 
$$
(f': U'_{V_{\eta}} \to V_{\eta} , \sL) \in \frM_h^{[N]}(V_{\eta}). 
$$
After replacing $\sL''$ by $(\sL'')^{|G|}$, so that $(\rho')^*\sL'' =(\sL''_{\overline V})^{|G|}$, 
we can ensure that the Hilbert polynomial of $\overline f$ with respect to $(\sL''_{\overline V})^{|G|}$ 
is equal to the one for $f''$ with respect to $\sL''$. 
In particular the induced moduli map  $\mu: V_{\eta}\to M_{h''}^{[N]}$ 
fits in the commutative diagram: 
\begin{equation}\label{diag:final}
\xymatrix{
V_{\eta}  \ar[drr]_{\mu}  &&  \overline V  \ar[ll]_{\sigma}  \ar[rr]^{\rho}  &&  V''  \ar[dll]^{\mu_{V''}} \\
&& M^{[N]}_{h''}
}
\end{equation}

Now, let $T\subseteq V_{\eta}$ be any subvariety. Assuming that $\mu_T:= \mu|_T$ is 
generically finite, from \autoref{diag:final} we find that $\rho|_{\sigma^{-1}(T)}$ is also generically finite. 
By \autoref{Kaw3} this implies that the Kodaira--Spencer map for $f_T$ is generically 
injective. 
Otherwise by \autoref{item:KSVar}, \autoref{def:KawVAR} and \autoref{item:trivial} 
through any general 
closed point of $V_{\eta}$ there is a subvariety $Y$ along which the Kodaira--Spencer 
map is zero and therefore by \autoref{Kaw3} $\rho$ must be constant along its preimage under $\sigma$, 
contradicting the generically finiteness of $\rho|_{\sigma^{-1}(T)}$.
Therefore, by \autoref{item:KSVar}, we have $\Var(f_T)= \dim T$.

On the other hand, if $\Var(f_T) = \dim T$, then again by \autoref{item:KSVar}, the Kodaira--Spencer 
map for $f_T$ is generically injective. In particular $\mu_T$ can only be generically finite. 
To conclude set $V^\circ:= V_{\eta}$.

\bibliographystyle{bibliography/skalpha} 


\end{document}
